\documentclass[11pt, a4paper, DIV=12, twoside=false, headings=big, BCOR=0mm, titlepage=false, toc=graduated]{scrartcl}

\typearea{last}
 \addtokomafont{sectioning}{\rmfamily}
 \usepackage[colorlinks=true, linkcolor=darkblue, citecolor=darkblue]{hyperref}

\usepackage{enumitem}		
\usepackage{amsmath} 		
\usepackage{amssymb}    
\usepackage{amsthm}     
\usepackage{mathrsfs} 	
\usepackage{color}      
	\definecolor{darkblue}{rgb}{0,0,0.5}
	\definecolor{darkgreen}{rgb}{0,0.5,0}	

\newcommand{\A}{{\mathcal A}}
\newcommand{\B}{{\mathcal B}}
\newcommand{\C}{{\mathcal C}}
\newcommand{\F}{{\mathcal F}}
\newcommand{\G}{{\mathcal G}}
\newcommand{\I}{{\mathcal I}}
	\let\acutedouble\H 
	\edef\ho{{\acutedouble{o}}}
	
\renewcommand{\H}{{\mathcal H}}
\newcommand{\X}{{\mathcal X}}

\newcommand{\N}{\mathbb{N^*}}
\newcommand{\No}{\mathbb{N}}
\newcommand{\R}{\mathbb{R}}
\newcommand{\Z}{\mathbb{Z}}

	\newcommand{\ReDeclareMathOperator}[2]{\renewcommand{#1}{\operatorname{#2}}}
\DeclareMathOperator{\Cov}{\mathbf{Cov}}

\DeclareMathOperator{\E}{\mathbf{E}}

\DeclareMathOperator{\id}{id}
\DeclareMathOperator{\ind}{\mathbf{1}}
\ReDeclareMathOperator{\L}{L}
\ReDeclareMathOperator{\P}{P}

\newcommand{\dconv}{\leadsto}

\newcommand{\lra}{\longrightarrow}

\newcommand{\ra}{\rightarrow}

\renewcommand{\epsilon}{\varepsilon}
\renewcommand{\phi}{\varphi}
\newcommand{\keywords}[1]{\noindent \textbf{Keywords:} #1} 
\newcommand{\MSC}[1]{\noindent \textbf{MSC:} #1} 

\newtheorem{theorem}{Theorem}
\newtheorem{lemma}{Lemma}
	
\newtheorem{proposition}{Proposition}
	
\newtheorem{corollary}{Corollary}
	
			\theoremstyle{definition}
\newtheorem*{definition}{Definition}

\newtheorem{assumption}{Assumption}
	
				\theoremstyle{remark}
\newtheorem{remark}{Remark}

\usepackage[sort&compress, longnamesfirst]{natbib}  

\begin{document}
\title{\LARGE\mdseries\scshape A Sequential Empirical CLT for Multiple Mixing Processes with Application to $\B$-Geometrically Ergodic Markov Chains%
\thanks{Research supported by the German Science Foundation, Grant DE 370-4 ({\em New Techniques for Empirical Processes of Dependent Data}).}
}

\author{\large
	Herold Dehling\thanks{Fakult\"at f¨ur Mathematik, Ruhr-Universit\"at Bochum, Universit\"atsstra\ss e 150, 44780 Bochum, Germany},
	Olivier Durieu\thanks{Laboratoire de Math\'ematiques et Physique Th\'eorique, UMR-CNRS 7350, F\'ed\'eration Denis Poisson FR-CNRS 2964, Universit\'e Fran\c{c}ois Rabelais de Tours, Parc
de Grandmont, 37200 Tours, France}, 
	Marco Tusche\footnotemark[2]\;\footnotemark[3]}

\date{September 15, 2014 }	

\maketitle

\begin{abstract}
\noindent \textbf{Abstract:} 
We investigate the convergence in distribution of sequential empirical processes of dependent data indexed by a class of functions $\F$. 
Our technique is suitable for processes that satisfy a multiple mixing condition on a space of functions which differs from the class $\F$.
This situation occurs in the case of data arising from dynamical systems or Markov chains, for which the Perron--Frobenius or Markov operator, respectively, has a spectral gap on a restricted space.
We provide applications to iterative Lipschitz models that contract on average.
\end{abstract}
	
	\smallskip
\MSC{60F05, 60F17, 60G10, 62G30, 60J05}\\ 
	\smallskip	
			\bigskip
\hrule
			\bigskip

\section{Introduction}\label{sec:intro}
The asymptotic behaviour of empirical processes has been studied for more than 60 years. The first rigorous result was the empirical process central limit theorem for i.i.d.\ data, established by \citet{Don52}. This theorem, conjectured by \citet{Doo49}, made it possible to derive the asymptotic distribution of a large number of test statistics and estimators that can be represented as functionals of the empirical process, by an application of the continuous mapping theorem. Among the examples are the Kolmogorov-Smirnov goodness of fit test, the Cram\'er-Von Mises $\omega^2$ criterion, and more generally von Mises statistics.

\citet{CieKes62} were among the first to extend Donsker's empirical process CLT to weakly dependent data, studying the empirical distribution of remainders in the dyadic expansion of a random number $\omega\in [0,1]$. \citet{Bil68} proved the first general result for dependent data, namely an empirical process CLT for data that can be represented as functionals of  a mixing process. For an overview of the literature on empirical processes of dependent data, see \citet{DehPhi02}, \citet{DedDouLan07}. 

\citet{Mul70}, and independently \citet{Kie72}, initiated the study of the sequential empirical process, defined as 
\begin{equation*}
 U_n(x,t)=\frac{1}{\sqrt{n}} \sum_{i=1}^{[nt]} \left( \ind_{\{X_i\leq x  \}} -F(x)  \right),
\end{equation*}
where $F(x)=P(X_1\leq x)$. The process $U_n(x,t)$ is also known as the two-parameter empirical process. Kiefer and M\"uller showed that for i.i.d.\ data, the sequential empirical process converges in distribution to a mean zero Gaussian process $K(x,t)$ with covariance structure
\begin{equation*}
\E\left( K(x,t) K(y,u) \right) = \min (t,u)(F(\min(x,y))-F(x) F(y)).
\end{equation*}
The limit process $K(x,t)$ is called Kiefer process, or Kiefer-M\"uller process.

\citet{KomMajTus75}, refining a technique originally invented by \citet{CsoRev75}, established the almost sharpest possible bounds for the error in the approximation of the sequential empirical process by the Kiefer process in the i.i.d.\ case so far.
For an overview of this topic, see the book by \citet{CsoRev81} or the survey article by \citet{GanStu79}.

Many authors have studied extensions of the sequential empirical process CLT to dependent data, e.g.\ \citet{BerPhi77} and \citet{PhiPin80} for strongly mixing processes and \citet{BerHorSch09} for S-mixing processes. 
Recently, \citet{DedMerRio13} proved strong approximation results for the sequential empirical process of some stationary sequences,
see also \citet{DedMerRio14} in the case of functions of absolutely regular sequences.
\citet{DehTaq89} determined the asymptotic distribution of the sequential empirical process in the case of long-range dependent data.

Recently, \citet{DehDurVol09} have developed a technique to prove empirical process CLTs for Markov chains and dynamical systems that do not necessarily satisfy any of the standard mixing conditions.
The technique has been extended by \citet{DehDur11}, \citet{DurTus12} and \citet{DehDurTus14a} to multivariate empirical processes and to empirical processes indexed by classes of functions. 
Among the examples that could be treated by the new techniques are $\B$-geometrically ergodic Markov chains,
dynamical systems with a spectral gap on the transfer operator and ergodic automorphisms of the $d$-dimensional torus, 
for which the empirical process CLT could be established.
It is the goal of the present paper to extend these techniques to the sequential empirical process, with a special focus on $\B$-geometrically ergodic Markov chain.
To this aim, we shall develop a sequential empirical CLT under multiple mixing (see definition in \autoref{def:mm}) that can be applied to this situation.

To illustrate our results, we present applications to a number of concrete examples. E.g., we establish a new sequential empirical process CLT for a class of Lipschitz models that contract on average; see Section \ref{Iterative}. We also present an application to ergodic torus automorphisms, and to expanding maps of the unit interval. These last two examples have recently also been investigated by  \citet{DedMerPen13} and by \citet{DedMerRio13}, who obtained results similar to ours. 

\medskip

Sequential empirical process CLTs  can be applied to the study of the asymptotic distribution of change-point tests based on the empirical distribution function.
Suppose $(X_i)_{i\in\No}$ is a stochastic process with marginal distribution functions
$\mu_1,\mu_2,\ldots$. Given the observations $X_1,\ldots,X_n$, we want to test the hypothesis
$\mathbf{H}_0$: \emph{``the process is stationary with marginal distribution $\mu$''} 
against the alternative
$\mathbf{H}_A$: \emph{``there exists a $k^* \in \{1,\ldots,n-1\}$ such that $(X_1,\ldots, X_{k^*})$ and $(X_{k^* +1},\ldots,X_n)$ are both stationary with different marginal distributions''}.
We propose the test statistic
\[
  T_n:=\max_{0\leq k\leq n} \sup_x \frac{k}{n} \biggl( 1-\frac{k}{n} \biggr) \sqrt{n}
  \bigl| F_{k}(x)-F_{k+1,n}(x)  \bigr|,
\]
where $F_k$ denotes the empirical distribution function of the observations $X_1,\ldots, X_k$ and $F_{k+1,n}$ denotes the empirical distribution function of $X_{k+1},\ldots,X_n$ (set $F_0=F_{n+1,n}=0$).
In order to determine the asymptotic distribution of $T_n$, we study the $\ell^\infty(\R\times[0,1])$-valued process $R_n= (R_n(x,t))_{(x,t)\in\R\times[0,1]}$ given by
\[
 R_n(x,t)=\sqrt{n} t(1-t) \bigl(F_{[nt]}(x) - F_{[nt]+1,n}(x)\bigr).
\]
As proved in \autoref{sec:stat}  (\autoref{the:motivation}), assuming ``convergence of the sequential empirical process'', we obtain under the null hypothesis $\mathbf{H}_0$ that
$$
R_n \dconv \bigl(K(x,t)-t K(x,1)\bigr)_{(x,t)\in\R\times[0,1]},
$$
where $K$ is the centred Gaussian process with covariance structure
\begin{align*}
&\Cov\bigl(K(x,t) , K(y,u)\bigr) \notag\\
\shoveright
	&= \min\{t,u\}\, \biggl\{ 
	\sum_{k=0}^{\infty} \Cov\bigl( \ind_{\{X_0 \leq x \}} , \ind_{\{X_k \leq y \}} \bigr) 
	+ \sum_{k=1}^{\infty} \Cov\bigl( \ind_{\{X_0 \leq y \}} , \ind_{\{X_k \leq x \}} \bigr)
	\biggr\}.
\end{align*}
This process is also referred to as a Kiefer process.
Applying the continuous mapping theorem to the supremum-functional, we obtain the asymptotic distribution of the test statistic $T_n$ under the null hypothesis, that is
\[
  T_n \dconv \sup_{x\in \R,\ t\in[0,1]} |K(x,t)-t K(x,1)|.
\]
Note that, in fact this result remains true for general $\F$-indexed empirical processes, (see \autoref{pro:test-statistic}).
\medskip

The remainder of this paper is organized as follows: 
In \autoref{sec:SECLT}, we recall some definitions and give the statement of a sequential empirical CLT for multiple mixing processes (\autoref{the:sq-ep-clt_mm}).
We also discuss an application of our general technique to the situation of the ergodic automorphisms of the torus.
In \autoref{sec:SECLT-spectralgap}, as application, we present sequential empirical CLTs for $\B$-geometrically ergodic Markov chains (\autoref{the:spectral-gap}) and dynamical systems with a transfer operator having a spectral gap (\autoref{the:dynamical-system}). A concrete application of \autoref{the:spectral-gap} to Lipschitz iterative models that contract on average (\autoref{thm:ite}) is also given in this section.
The asymptotic distribution of the test statistic $T_n$ (\autoref{pro:test-statistic}) is given in \autoref{sec:stat}.
The proofs of the main results are postponed to \autoref{sec:proof_SECLT} and \autoref{sec:proof_covariance}.


\section{A Sequential Empirical CLT for Multiple Mixing Processes}\label{sec:SECLT}

\subsection{Definitions and Notations}\label{sec:def}
Let $(\X,\A)$ be a measurable space.
For a positive measure $\lambda$ on $\X$ and a $\lambda$-integrable complex-valued function $f$ on $\X$, we will use the notation $\lambda f:= \int_\X f ~d\lambda$.
For $s\in[1,\infty)$, we denote by $\L^s(\lambda)$ the Lebesgue space of $s$-th power integrable complex-valued functions on $\X$.
This space is equipped with the norm $\|f\|_s=(\lambda (|f|^{s}))^{1/s}$. 
Further, we denote the space of essentially bounded measurable functions on $\X$ w.r.t.\ $\lambda$ by $L^\infty(\lambda)$ and the corresponding (essential) supremum norm by $\|\cdot\|_\infty$.
Note that these norms depend heavily on the choice of the measure $\lambda$; however throughout this paper it will always be clear which measure we refer to.

Let $(X_i)_{i\in\No}$ be an $\X$-valued stationary stochastic process with marginal distribution $\mu$ and
let $\F$ be a class of real-valued measurable functions on $\X$ which is uniformly bounded w.r.t.\ the $\|\cdot\|_\infty$-norm.
For $n\in\N$, we define the map $F_n:\F \lra \R$, induced by the empirical measure, by
\[
F_n(f) := \frac{1}{n} \sum_{i=1}^{n} f(X_i),\quad f\in\F.
\]
The \emph{sequential empirical process} of the $n$-th order of $(X_i)_{i\in\No}$ is then the $\F\times[0,1]$-indexed process $U_n:=(U_n(f,t))_{(f,t)\in\F\times[0,1]}$
given by
\begin{align*} 
	U_n(f,t) := \frac{[nt]}{\sqrt{n}} \bigl(F_{[nt]}(f) - \mu f\bigr)
	= \frac{1}{\sqrt{n}}\sum_{i=1}^{[nt]} \bigl( f(X_i) - \mu f \bigr),  \quad(f,t)\in\F\times[0,1],
\end{align*}
where $[\cdot]$ denotes the lower Gauss bracket, i.e.\ $[x]:=\sup\{z\in\Z : z\leq x\}$.

For fixed $n\in\N$, we consider $U_n$ as a random element in the metric space $\ell^\infty(\F\times[0,1])$ of bounded real-valued functions on $\F\times[0,1]$, equipped with the supremum norm and the corresponding Borel $\sigma$-field. 
Since $\F\times[0,1]$ is uncountable, here we cannot assume that $U_n$ is measurable and thus standard techniques of weak convergence do not apply.
We will therefore use the theory of outer probability and expectation (see \citet{VanWel96}). 

Let $\E^* X$ denote the outer expectation of a possibly non-measurable random element $X$,
let $U$ be measurable, and let $U,U_0,U_1,\ldots$ take values in $\ell^\infty(\F\times[0,1])$. 
We define convergence in distribution or weak convergence $U_n\dconv U$ in $\ell^\infty(\F\times[0,1])$ as the convergence
$ \E^*(\phi(U_n)) \ra \E(\phi(U))$
of all bounded and continuous functions $\phi: \ell^\infty(\F\times[0,1]) \lra \R$. 
We say that the process $(X_i)_{i\in\No}$ satisfies a sequential empirical CLT if the process $U_n$ converges in distribution in $\ell^\infty(\F\times[0,1])$ to a tight centred Gaussian process.

Empirical CLTs usually require some bound of the size of the indexing class $\F$. 
This size is usually measured by counting certain sets, e.g.\ balls or brackets of a given $\|\cdot\|_s$-size, needed to cover $\F$ (c.f.\ \citet{Oss87} and \citet[p.83 ff.]{VanWel96}).
In our upcoming setting, we will only deal with properties for functions of a restricted class which could be disjoint of the class $\F$. We thus need an adapted notion of bracketing numbers. This notion was introduced in \citet*{DehDurTus14a}.
\begin{definition}\label{def:bracket}
Let $(\X,\A,\mu)$ be a probability space.
For two functions $l, u:\X\ra\R$ such that $l(x)\leq u(x)$ for all $x\in\X$, 
we define the bracket
\[
  [l,u]:=\{f:\X\rightarrow \R: l(x) \leq f(x) \leq u(x), \mbox{ for all } x\in \X  \}.
\]
Let $\G$ be a subset of a normed real vector space $(\C,\|\cdot\|_\C)$ of measurable real-valued functions on $\X$.
For given $\epsilon>0$, $A>0$, and $s\in[1,\infty]$, we call
$[l,u]$ an $(\epsilon,A,\G,\L^s(\mu))$-bracket, if $l,u\in \G$ and
\begin{align*}
 &\|u-l\|_s \leq \epsilon \\
 &\|u\|_\C \leq A,\quad \|l\|_\C \leq A.
\end{align*}
For a class of real-valued functions $\F$ on $\X$, we define the bracketing number
\[
	N(\epsilon,A,\F,\G,\L^s(\mu))
\]
as the smallest number of $(\epsilon, A, \G, \L^s(\mu))$-brackets needed
to cover $\F$.
\end{definition}
This notion of brackets allows to control the number of brackets needed to cover $\F$ not only with respect to the decreasing size of the brackets in $\L^s$-norm, but also with a control of the increasing $\|\cdot\|_\C$-size of the bracketing functions as the $\L^s$-norm goes to zero.


\subsection{Multiple mixing processes and the main result}
In this section, we present a general result which will be applied to $\B$-geometrically ergodic Markov chains in \autoref{sec:SECLT-spectralgap}.
We consider stationary sequences $(X_i)_{i\in\No}$ which satisfy a multiple mixing condition with respect to some space of functions. Let  $(\C,\|\cdot\|_\C)$ be some normed vector space of functions on $\X$. The multiple mixing condition is defined as follows.
\begin{definition}[Multiple Mixing Processes] \phantomsection\label{def:mm}
We say that a process $(X_i)_{i\in\No}$ is \emph{multiple mixing} with respect to $\C$ if
there exist a real $\theta\in (0,1)$, a real $s\ge 1$, and an integer $d_0\in\No$ such that for all $p\in\N$, there exist an integer $\ell$ and a multivariate polynomial $P$ of total degree not larger than $d_0$ such that
\begin{equation}\label{mm}
  \left|  \Cov(f(X_{i_0})\cdots f(X_{i_{q-1}})  ,f(X_{i_q})\cdots f(X_{i_p}) )   \right| \leq  \|f\|_s \|f\|_{\C}^\ell P(i_1-i_0,\ldots,i_p-i_{p-1})\theta^{i_q-i_{q-1}}
\end{equation}
holds for all $f\in \C$ with $\mu f=0$ and $\|f\|_\infty \leq 1$,
all integers $i_0\le i_1 \le \ldots \le i_p$ and all $q\in \{1,\ldots,p\}$. 
\end{definition}
As proved in \citet{DehDur11}, multiple mixing processes satisfy a moment bound which is particularly useful to establish empirical CLTs.

\medskip

The approach developed here is useful when the indexing class $\F$ is different from the space $\C$. 
In the following we shall require the two following assumptions concerning the processes $(f(X_i))_{i\in\No}$, where $f:\X\lra\R$ belongs to $(\C,\|\cdot\|_\C)$. 
%
\begin{assumption}[Finite dimensional sequential CLT for $\C$-observables]\label{asp:clt_seq} 
For every choice of $f_1,\ldots,f_k\in\C$ and $t_1,\ldots,t_k\in[0,1]$
\[
 \frac{1}{\sqrt{n}}\left( \sum_{i=1}^{[nt_1]}(f_1(X_i)-\mu f_1) \;,\; \ldots \;,\; \sum_{i=1}^{[nt_k]}(f_k(X_i)-\mu f_k) \right) \dconv N(0,\Sigma),
\]
where $N(0,\Sigma)$ denotes some $k$-dimensional normal distribution with mean zero and covariance matrix $\Sigma=(\Sigma_{i,j})_{1\leq i,j\leq k}$.
\end{assumption}

\begin{assumption}[Multiple mixing w.r.t.\ $\C$]\label{asp:mm}
The process $(X_i)_{i\in\No}$ is multiple mixing with respect to $\C$, and with parameters $\theta\in(0,1)$, $s\ge 1$, and $d_0\in\No$.
\end{assumption}


To derive a CLT for an $\F$-indexed empirical process, we now have to precise the relation between the class $\F$ and the space $\C$.
Note that, in the particular case where $\F$ is a subset of $\C$, from \autoref{asp:clt_seq} we can infer the finite dimensional convergence of the process $(U_n)_{n\in\No}$. Then, the tightness can be established under an entropy condition on $\F$ that uses the usual bracketing number defined as in \citet{Oss87}.
Nevertheless, in many examples, the functions of $\F$ do not belong to the space $\C$. To overcome this difficulty, we have to measure how the functions of $\F$ are well approximated by the functions of $\C$.
We will use the bracketing numbers introduced in the preceding section to 
obtain a control on the size of $\F$ which depends on the possibility of approximation by the space $\C$.

\medskip

We can show the following sequential empirical CLT.

\begin{theorem}\label{the:sq-ep-clt_mm}
Let $(\X,\A)$ be a measurable space, let $(X_i)_{i\in\No}$ be an $\X$-valued stationary process with marginal distribution $\mu$, and let $\F$ be a uniformly bounded class of measurable functions on $\X$.
Suppose that, for some normed vector space $\C$ of measurable functions on $\X$, \autoref{asp:clt_seq} and \autoref{asp:mm} hold.

If there exist a subset $\G$ of $\C$ which is bounded in $\|\cdot\|_\infty$-norm, 
$C>0$, $r>-1$, and $\gamma>d_0+1$ such that
			\begin{equation}
				\label{con:entropy_exp}
				\int_0^1 \varepsilon^{r}\sup_{\epsilon \le \delta \le 1} N^2\bigl(\delta,
				\exp\bigl(C \delta^{-\frac{1}{\gamma}}\bigr),\F,\G,\L^s(\mu)\bigr)
				d\varepsilon <\infty
			\end{equation}
then the sequential empirical process $U_n$ converges in distribution in 
$\ell^\infty(\F\times[0,1])$ to a tight Gaussian process $K$.
\end{theorem}

Observe that for $r'\ge 0$, inequality \eqref{con:entropy_exp} holds for all $r>2r'-1$, if 
\[ N\bigl(\epsilon,\exp\bigl(C \delta^{-\frac{1}{\gamma}}\bigr),\F,\G,\L^s(\mu)\bigr) = O(\varepsilon^{-r'})\quad \text {as}\ \varepsilon\rightarrow 0.\]
Note further, that the supremum in \eqref{con:entropy_exp} appears in order to deal with the possible non-monotonicity of the bracketing number.

Let us also mention that several classes of functions $\F$ which satisfy the condition \eqref{con:entropy_exp} with respect to a space of bounded H\"older functions are presented in \citet{DehDurTus14a}. Among these classes are indicators of rectangles, indicators of balls, indicators of ellipsoids, and a class of monotone functions in dimension 1.

\medskip

In this general setting of \autoref{the:sq-ep-clt_mm}, we are unable to specify the covariance structure of the limit process. 
The next corollary shows that under additional conditions, the limit process of $U_n$ is indeed a Kiefer process.

\begin{corollary}\label{lem:covariance_seq} 
In the situation of \autoref{the:sq-ep-clt_mm}, assume further that 
\begin{enumerate}[label=(\roman*), ref=\roman*]
	\item\label{con:clt_cov}
	\autoref{asp:clt_seq} holds with covariance matrix $\Sigma$ given by 
 	\begin{align} 
 		\Sigma_{i,j}
 		= \min\{t_i,t_j\}\, \biggl\{ 
 		\sum_{k=0}^{\infty} \Cov\bigl( f_i(X_0), f_j(X_k) \bigr) 
 		+ \sum_{k=1}^{\infty} \Cov\bigl( f_j(X_0) , f_i(X_k) \bigr) \biggr\}, \label{eq:clt_cov}
 	\end{align}
	\item\label{con:bimix}
there exists a constant $D>0$ such that for all $f\in\G\cup(\G-\G)$ and all $\phi\in\F\cup(\F-\G)$  
 	\begin{equation}\label{eq:cov_phi-f}
\bigl| \Cov\bigl( \phi(X_0), f(X_k)\bigr) \bigr| \leq D \|\phi\|_\infty \| f \|_\C \theta^k, 
  	\end{equation}
\end{enumerate}
Then the covariance structure of the limit process $K$ is given by 
 	\begin{align}
	 		&\Cov\bigl(K(f_1,t_1),K(f_2,t_2)\bigr) \notag\\
 		\shoveright&= \min\{t_1,t_2\}\, \biggl\{ 
 	 \sum_{k=0}^{\infty} \Cov\bigl( f_1(X_0), f_2(X_k) \bigr) 
 		+ \sum_{k=1}^{\infty} \Cov\bigl( f_1(X_k) , f_2(X_0) \bigr) \biggr\}, \label{eq:cov-structure}
	 \end{align}
for all $ f_1,f_2\in\F$, $t_1,t_2\in[0,1]$.
\end{corollary}
The proof of \autoref{the:sq-ep-clt_mm} and \autoref{lem:covariance_seq}  are given, respectively, in \autoref{sec:proof_SECLT} and \autoref{sec:proof_covariance}.

\begin{remark}\label{rem:kiefer}
A centred Gaussian process $K$ with covariance structure \eqref{eq:cov-structure} is often referred to as a Kiefer process.
\end{remark}


In \autoref{sec:SECLT-spectralgap}, we will apply \autoref{the:sq-ep-clt_mm} to prove a sequential empirical CLT for $\B$-geometrically ergodic Markov chains, which is the main motivation of the paper. Before, we would like to mention that other applications of \autoref{the:sq-ep-clt_mm} are possible.

\paragraph{Ergodic Automorphism of the Torus}

Let $T$ be an ergodic automorphism of the $d$ dimensional torus $\mathbb{T}^d$ as introduced in Section 4 of \citet{DehDurTus14a}.
Following the ideas of \citet{DehDurTus14a}, we can extend their theorem to a sequential empirical CLT.
Let $\G$ be a bounded subset of $\H_{\alpha}(\mathbb{T}^d,\R)$, $\alpha\in(0,1]$, let $\mu=\lambda$ be the Lebesgue measure on $\mathbb{T}^d$ and 
assume further that $\F$ is a uniformly bounded class of functions from $\mathbb{T}^d$ to $\R$. We denote by $d_0$ the size of the biggest Jordan block of $T$ restricted to its neutral subspace.
We can establish the following result which is not proved here.
\begin{corollary}
Assume that the class $\F$ satisfies the condition \eqref{con:entropy_exp} for some $\gamma>d_0+1$. Assume further that there exist $C>0$ and $a>0$ such that for all $f\in\F$ and $k\in\N$, there exists $g_k\in\G$ satisfying $\|f-g_k\|_1\le k^{-1}$ and $\|g_k\|\le C k^a$.
Then the sequential empirical process given by 
\[ U_n(f,t)=\frac{1}{\sqrt{n}}\sum_{i=1}^{[nt]}(f\circ T^i - \lambda f),\quad f\in\F, t\in[0,1]\]
converges in distribution in $\ell^\infty(\F\times[0,1])$ to a Kiefer process.
\end{corollary}
Note that both assumptions on $\F$ are satisfied e.g.\ if $\F$ is the class of indicators of rectangles, balls, or ellipsoids (to see this, follow the proof of Proposition 3.2, 3.5 and 3.6 in \citet{DehDurTus14a}). 

This result is proved in details in \citet{Tus14} by application of \autoref{the:sq-ep-clt_mm}. We just notice here that, in this situation, the multiple mixing property holds (see \citet{DehDur11}) and that \autoref{asp:clt_seq} can be derived from the classical CLT (see \citet{Tus14}, Lemma 11.1).
Further, assumption \eqref{con:bimix} of \autoref{lem:covariance_seq} is not straightforward. Instead, using the second assumption of the proposition,
we can show that there exist some $c>0$ and $\theta \in (0,1)$ such that for all $f\in\F$ and $g\in\H_\alpha(\mathbb{T}^d,\R)$, $|\Cov(f,g\circ T^n)|\le c \|g\|_\alpha \theta^n$, which is sufficient to conclude as in \autoref{lem:covariance_seq}.

\begin{remark}
As mentioned in the introduction, for ergodic torus automorphisms \citet{DedMerPen13} have investigated the sequential empirical process indexed by a class of the form 
$\{1_{(-\infty,t]}\circ f : t\in\R^l\}$, where $f:\mathbb{T}^d\to \R^l$ is fixed. Under some regularity assumptions on $f$, and using techniques different from ours, \citet{DedMerPen13} obtain weak convergence to a Kiefer process. They also develop a tightness criterion (Proposition 3.13) that can be applied to many other examples, e.g. those given in \citet{DedPri07}.
\end{remark}

\paragraph{Multiple Mixing of Lower Rate}

Processes of a lower mixing rate have been studied by \citet{DurTus12}. They consider a multiple mixing condition w.r.t.\ the space of bounded $\alpha$-H\"older functions on $\R^d$, $\alpha\in(0,1]$, where the term $\theta^{i_q-i_{q-1}}$ in \eqref{mm} is replaced by a general term $\Theta({i_q-i_{q-1}})$ with a monotone decreasing function $\Theta:\No\ra\R_+$.
Under the condition that $\sum_{i=0}^\infty i^{2p-2}\Theta(i) <\infty$, they were able to establish an empirical CLT. 
This could also be extend to a sequential version. Since it is not needed for our application, we decide to not develop this very general setting here.
We can just remark that, in this situation, a stronger entropy condition will be needed. In particular, the second parameter in the bracketing number which appears in \eqref{con:entropy_exp} should be replaced by a polynomial function of $\varepsilon^{-1}$.


\section{A Sequential Empirical CLT under Spectral Gap}\label{sec:SECLT-spectralgap}

We now present an application of \autoref{the:sq-ep-clt_mm} to establish a sequential empirical CLT for Markov chain having a spectral gap property.

\subsection{\texorpdfstring{$\B$-geometrically ergodic Markov chains}{B-geometrically ergodic Markov chains}}
\label{sec:sqclt-mc}
\newcounter{count:con-sg}

In the following, let $(X_i)_{i\in\No}$ be a time homogeneous Markov chain on a measurable state space $(\X,\A)$ with a probability transition $P$ and an invariant measure $\nu$.
We assume that the Markov chain starts with initial distribution 
$\nu$, i.e\ that the distribution of $X_0$ is $\nu$. This makes $(X_i)_{i\in\No}$ a stationary sequence.
We also denote by $P$ the associated Markov operator defined by
\[
	Pf = \int_\X f(y) ~P(\cdot,dy).
\]

Now, let $(\B,\|\cdot\|_\B)$ a Banach space of measurable functions from $\X$ to $\R$. We will assume that $P$ is a bounded linear operator on $\B$, and we denote by $\mathcal{L}(\B)$ the space of all bounded linear operators from $\B$ to $\B$.

We will need the following properties on the Banach space $\B$:
\begin{enumerate}[label=(\Alph*), ref=\Alph*, series=sg-conditions]
	\item\label{con:in_B} 
	   $\ind_\X \in\B$ and $P\in\mathcal{L}(\B)$.
\end{enumerate}	
For some $m\in [1,\infty]$,
\begin{enumerate}[resume*=sg-conditions]	
	\item\label{con:B->Lm}
	$\B$ is continuously included in $L^m(\nu)$, i.e.\ there is a $K>0$ such that $\|\cdot\|_m \leq K \|\cdot\|_\B$.
\end{enumerate}
Further we consider processes such that the action of the corresponding Markov operator on $\B$ satisfies
\begin{enumerate}[resume*=sg-conditions]
	\item\label{con:sg}
		$\|P^nf- (\nu f) \ind_\X\|_\B \leq \kappa \|f\|_\B \theta^n$ 
	for some $\kappa>0$, $\theta\in[0,1)$, and all $f\in\B$.
\end{enumerate}
This property is often referred to as strong or geometric ergodicity with respect to $\B$ (c.f.\ \citet{MeyTwe93}, \citet{Her08}, and \citet{HerPen10}). 
\begin{remark}
Note that condition \eqref{con:sg} corresponds to a spectral gap property of $P$ acting on $\B$, i.e.\ 
$1$ is the only eigenvalue of modulus one, it is simple, and the rest of the spectrum is contained in a disk of radius strictly smaller than one. Further, in this case there exists a decomposition of the linear operator $P$ in $\mathcal{L}(\B)$, 
\begin{align*}
P=\Pi +N, 
\end{align*}
such that $\Pi f= (\nu f)\ind_\mathcal{X}$ is a projection on the eigenspace of $1$, $N\circ \Pi=\Pi\circ N=0$, and $\rho(N):=\lim_{n\to\infty}\|N^n\|^{1/n}_{\mathcal{L}(\B)}<1$,
where $\|\cdot\|_{\mathcal{L}(\B)}$ denotes the operator norm on $\B$.
\end{remark}
We first show below that the conditions \eqref{con:in_B} -- \eqref{con:sg} guarantee a sequential finite dimensional CLT for functions in $\B$.

Actually, we will establish a $k$-dimensional Donsker invariance principle, 
which of course implies the desired result by a projection.

\begin{proposition}\label{pro:mc_fidi-clt}
Suppose that for some $m\in[1,\infty]$, \eqref{con:in_B}, \eqref{con:B->Lm}, \eqref{con:sg} hold. Let $k$ be a positive integer and $f_1,\dots,f_k \in\B\cap L^s(\nu)$, with $s=m/(m-1)$.
Then 
\begin{align}\label{eq:psp-conv}
\bigl(U_n(f_1,t),\ldots,U_n(f_k,t)\bigr)_{t\in[0,1]} \dconv W 
\end{align}
in $(\ell^\infty[0,1])^k$, where $W:=\bigl(W_1(t),\ldots,W_k(t)\bigr)_{t\in[0,1]}$ is a centred Gaussian process with covariances  
\begin{align*} 
 		\Cov\bigl(W_i(t),W_j(u)\bigr)
 		= \min\{t,u\}\, \biggl\{ 
 	 \sum_{k=0}^{\infty} \Cov\bigl( f_i(X_0), f_j(X_k) \bigr) 
 		+ \sum_{k=1}^{\infty} \Cov\bigl( f_j(X_0) , f_i(X_k) \bigr) \biggr\}. 
 	\end{align*}	
\end{proposition}
In particular this proposition shows that \autoref{asp:clt_seq} holds with covariance structure \eqref{eq:clt_cov}.

\begin{proof}
To prove this proposition, we will use a result of \citet{DedMer03}. 
An application of their Corollary~2 (see also Theorem~2 in \citet{DedMer03}) yields that a sufficient condition for the convergence \eqref{eq:psp-conv} is that the centred random vector $Z_i:=\bigl(f_1(X_i)-\nu f_1,\ldots,f_k(X_i)-\nu f_k\bigr)$ satisfies 
\begin{equation}\label{ded}
\sum_{i=0}^\infty \E\bigl(\bigl|Z_0 ||\E(Z_i|X_0)\bigr|\bigr)<\infty.
\end{equation}
Here $|\cdot|$ denotes the Euclidean norm on $\R^k$. By H\"olders inequality one has
\[
\E\bigl(\bigl|Z_0 ||\E(Z_i|X_0)\bigr|\bigr) \le \E(|Z_i|^s)^\frac{1}{s} \E\bigl(\bigl|\E(Z_i|X_0)\bigr|^m\bigr)^\frac{1}{m}.
\]
The assumption that the $f_j$ belong to $L^s(\nu)$ gives that $\E(|Z_i^s|)^\frac1s<\infty$. Applying \eqref{con:B->Lm} and \eqref{con:sg}, we finally obtain 
\[
\E\bigl(\bigl|\E(Z_i|X_0)\bigr|^m\bigr)^\frac1m
\le \sum_{j=1}^k\bigl\|\E(f_j(X_i)-\nu f_j |X_0)\bigr\|_m
\le K\sum_{j=1}^k\|P^if_j-\nu f_j\|_\B \le  K\kappa\sum_{j=1}^k\|f_j\|_\B \theta^i,
\]
which shows that \eqref{ded} holds and thus proves the proposition.
\end{proof}

Note that without the assumption that the $f_i$ belong to $L^s(\nu)$, it is still possible to prove a finite-dimensional sequential CLT (\autoref{asp:clt_seq}) using the Nagaev method consisting of operator perturbations. However, without $f_i\in L^s(\nu)$ we do not obtain a characterization of the covariance matrix (see \citet{Tus14} for details).

Now, to apply \autoref{the:sq-ep-clt_mm} to prove a sequential empirical CLT, we need to show the {multiple mixing} property of $(X_i)_{i\in\No}$. To this aim, the following further condition on the space $\B$ is useful. 
\begin{enumerate}[resume*=sg-conditions]
\item \label{con:norm} There exist $C>0$ and $\ell\in\N$ such that, if $f\in\B$ and $g\in\B$ are bounded by $1$, then $fg\in\B$ and 
                  $\|fg\|_\B\le C\max\{\|f\|_\B,\|g\|_\B\}^\ell$.  
\end{enumerate}
Note that if $\B$ is a Banach algebra, condition \eqref{con:norm} holds with $\ell=2$.

The following lemma is now a straightforward extension of Lemma 3 in \citet{DehDur11}.
\begin{lemma}\label{lem:mm}
Under the conditions \eqref{con:in_B}, \eqref{con:B->Lm}, \eqref{con:sg}, and \eqref{con:norm}, 
$(X_i)_{i\in\No}$ satisfies the {multiple mixing} property  w.r.t.\ $\B$ with  $d_0=0$ and $s=m/(m-1)$.
\end{lemma}

Eventually, observe that the second assumption of \autoref{lem:covariance_seq} is also satisfied as it is shown by the following lemma.

\begin{lemma}\label{l:cov}
Under the conditions \eqref{con:in_B}, \eqref{con:B->Lm}, and \eqref{con:sg}, for all $f\in\B$ and all $g\in \L^s(\nu)$, with $s=\frac{m}{m-1}$, we have
\[
  |\Cov(g(X_0), f(X_k))|\le C \|g\|_s\|f\|_\B \theta^k.
\]
\end{lemma}
\begin{proof}
Applying successively H\"older's inequality, \eqref{con:B->Lm}, and  \eqref{con:sg}, we get
\begin{align*}
|\Cov(g(X_0), f(X_k))|
&\le \E |g(X_0) \E(f(X_n)-\nu f | X_0) |\\
& \le \|g\|_s\|P^kf-(\nu f)\ind_\X\|_\B\\
& \le C \|g\|_s\|f\|_\B \theta^k.
\end{align*}
\end{proof}

As a conclusion, we thus have the following sequential empirical central limit theorem as a corollary of \autoref{the:sq-ep-clt_mm}, \autoref{lem:covariance_seq}, \autoref{pro:mc_fidi-clt}, and \autoref{lem:mm}.

\begin{theorem}[Sequential empirical CLT for $\B$-geometrically ergodic Markov chains]\label{the:spectral-gap}
Let $\F$ be a $\|\cdot\|_\infty$-bounded class of functions from $\X$ to $\R$.
Assume that for some $m\in[1,\infty]$, the conditions \eqref{con:in_B}, \eqref{con:B->Lm}, \eqref{con:sg}, and \eqref{con:norm} hold.
If there is a $\|\cdot\|_\infty$-bounded subset $\G\subset\B$ such that \eqref{con:entropy_exp} is satisfied with $s=m/(m-1)$,
then the sequential empirical process converges in distribution in $\ell^\infty(\F\times [0,1])$ to a centred Gaussian process $K$ with covariance structure given by \eqref{eq:cov-structure}.
\end{theorem}


Now, let us give an example by applying \autoref{the:spectral-gap} to random iterative Lipschitz models.

\subsection{Iterative Lipschitz models that contract on average}\label{Iterative}
	
In this section, we assume that $(\X,d)$ is a (not necessarily compact) metric space in which every closed ball is compact. 
Further we assume that $\X$ is equipped with the Borel $\sigma$-algebra $\mathfrak{B}(\X)$. Let $\{T_i,\ i\ge 0\}$ be a family of Lipschitz maps from $\X$ to $\X$. We consider the Markov chain with state space $\mathcal{X}$ and transition probability $P$ given by
\[
 P(x,A)=\sum_{i\ge 0}p_i(x)\ind_{A}(T_i(x)),\quad x\in\mathcal{X},\ A\in\mathfrak{B}(\X),
\] 
where the $p_i$ are Lipschitz functions from $\X$ to $[0,1]$ which satisfy $\sum_{i\ge 0}p_i(x)=1$ for all $x\in\X$. 
Thus, each step of the Markov chain corresponds to the application of one of the maps $T_i$ which is chosen randomly
 with respect to a probability distribution which depends on the actual state of the chain. 
We assume that this model has a property of contraction on average, that is that there exists a $\rho\in(0,1)$ such that
\begin{equation}\label{ca}
 \sum_{i\ge 0}d(T_i(x),T_i(y))p_i(x) < \rho d(x,y),\quad \forall x, y\in\X.
\end{equation}

Statistical properties of such models have been studied by \citet{DubFre66}, \citet{BarElt88}, \citet{HenHer01}, \citet{WuSha04}, \citet{Her08}, and by \citet{HerPen10} in the case of constant functions $p_i$ and by \citet{DoeFor37}, \citet{Kar53}, \citet{BarDemEltGer88}, \citet{Pei93}, \citet{Pol01}, and by \citet{Wal07} in the case of variable functions $p_i$. 

As in many of the cited papers, we need the following technical properties. For some fixed $x_0\in\X$, suppose
\begin{align}
\label{h0}
&\sup_{\substack{x,y,z\in\mathcal{X},\\ y\ne z}}\sum_{i\ge 0}\frac{d(T_i(y),T_i(z))}{d(y,z)}p_i(x)<\infty,
\\
&\label{h1}
\sup_{x,y\in\mathcal{X}}\sum_{i\ge0}\frac{d(T_i(y),x_0)}{1+d(y,x_0)}p_i(x) <\infty,
\\
&\label{h2}
\sup_{x\in\mathcal{X}}\sum_{i\ge 0}\frac{d(T_i(x),x_0)}{1+d(x,x_0)}\sup_{y,z\in\mathcal{X}, y\ne z}\frac{|p_i(y)-p_i(z)|}{d(y,z)}<\infty. 
\end{align}
Moreover assume that for all $x, y\in\mathcal{X}$, there exist sequences of integers $(i_n)_{n\ge1}$ and $(j_n)_{n\ge 1}$ such that
\begin{equation}\label{h4}
 d\bigl(T_{i_n}\!\circ\ldots\circ T_{i_1}(x)\,,\,T_{j_n}\!\circ\ldots\circ T_{j_1}(y)\bigr)\bigl(1+d\bigl(T_{j_n}\!\circ\ldots\circ T_{j_1}(x)\,,\,x_0\bigr)\bigr)\to 0\quad \text{as}\ n\to\infty
\end{equation}
with $p_{i_n}(T_{i_{n-1}}\!\circ\ldots\circ T_{i_1}(x))\cdot\ldots\cdot p_{i_1}(x)>0$ and  $p_{j_n}(T_{j_{n-1}}\!\circ\ldots\circ T_{j_1}(y))\cdot\ldots\cdot p_{j_1}(x)>0$.
Note that conditions \eqref{h0} -- \eqref{h2} are verified when the family of maps $T_i$ is finite and 
\eqref{h4} is verified when \eqref{ca} -- \eqref{h2} hold and each $p_i$ is positive.
See \citet{Pei93} for a discussion on these assumptions.

Under the conditions \eqref{ca} -- \eqref{h4}, \citet{Pei93} proved that the Markov chain has an attractive $P$-invariant probability measure $\nu$ with existing first moment. We define the stationary process $(X_i)_{i\in\No}$ on $\X$ as the Markov chain with transition probability $P$ starting with distribution $\nu$, that is $X_0\sim \nu$.

A central limit theorem for the empirical process associated to the Markov chain $(X_i)_{i\ge 0}$ was proved by \citet{Dur13} (see also \citet{WuSha04} in the case of constant functions $p_i$). The following theorem extends this result to the sequential empirical processes.

For $\alpha\in(0,1]$, we consider the space $\H_\alpha(\X)$ of bounded $\alpha$-H\"older continuous functions on $\X$ with values in $\R$, equipped with the norm 
\[
	\|\cdot \|_{\H_\alpha} := \|\cdot\|_\infty + m_\alpha(\cdot),
\]
where 
\[
m_\alpha(f):=\sup_{\substack{x,y\in\X \\ x\neq y}} \frac{|f(x)-f(y)|}{d(x,y)^\alpha}.
\]

\begin{corollary}\label{thm:ite}
Let \eqref{ca} -- \eqref{h4} hold, $(X_i)_{i\in\No}$ be the Markov chain with transition probability $P$ starting under the invariant distribution $\nu$, and
consider a $\|\cdot\|_\infty$-bounded class of functions $\F$.
Let $s\in(1,2)$ and $\G$ be a $\|\cdot\|_\infty$-bounded subset of the space $\H_\alpha(\X)$ for some $\alpha<\frac{s-1}{s}$ such that \eqref{con:entropy_exp} holds.  
Then the $\F$-indexed sequential empirical process $(U_n(f,t))_{\F\times[0,1]}$ associated to the process $(X_i)_{i\ge 0}$ converges in distribution in the space $\ell^\infty(\F\times [0,1])$ to a centred Gaussian process with covariance given by \eqref{eq:cov-structure}. 
\end{corollary}

\begin{proof}
First, we introduce spaces of Lipschitz functions with weights that give the geometric ergodicity of the chain.
For every $\alpha,\beta\in[0,1]$, let $\H_{\alpha,\beta}(\X)$ denote the space of continuous function from $\X$ to $\R$ with
$\|f\|_{\H_{\alpha,\beta}}=N_\beta(f)+m_{\alpha,\beta}(f)<\infty$, where
\[
 N_\beta(f)=\sup_{x\in\mathcal{X}}\frac{|f(x)|}{1+d(x,x_0)^\beta} 
\quad\text{and}\quad 
m_{\alpha,\beta}(f)=\sup_{x,y\in\mathcal{X}, x\ne y}\frac{|f(x)-f(y)|}{d(x,y)^\alpha(1+d(x,x_0)^\beta)} .
\]
In particular, the space $\H_\alpha(\X):=\H_{\alpha,0}(\X)$ is the space of bounded $\alpha$-H\"older functions from $\X$ to $\R$
and we have $\|\cdot\|_{\H_{\alpha,0}}=2^{-1}\|\cdot\|_{\H_\alpha}$.
It is a subspace of $\H_{\alpha,\beta}(\X)$ for all $\beta>0$. 
The following properties are straightforward and given without proof.
\begin{lemma}\label{l:H}
For all $\alpha$ and $\beta\in[0,1]$,
 \begin{enumerate}[label=(\roman*), ref=\roman*]
  \item \label{i}the space $(\H_{\alpha,\beta}(\X),\|\cdot\|_{\H_{\alpha,\beta}})$ is a Banach space which satisfies condition \eqref{con:in_B},
  \item \label{ii} for every bounded functions $f,g\in\H_{\alpha,\beta}(\X)$, we have that\\ $\|fg\|_{\H_{\alpha,\beta}}\le \|f\|_\infty\|g\|_{\H_{\alpha,\beta}}+\|g\|_\infty\|f\|_{\H_{\alpha,\beta}}$,
 \item \label{iii} for every $f\in\H_\alpha(\X)$ and $g\in\H_{\alpha,\beta}(\X)$, we have that $\|fg\|_{\H_{\alpha,\beta}}\le \|f\|_{\H_\alpha}\|g\|_{\H_{\alpha,\beta}}$,
 \item \label{iv}  there exists a $C>0$, for every $f\in\H_{\alpha,\beta}(\X)$, $f\in L^\frac1\beta(\nu)$ and 
         $\|f\|_\frac1\beta \le C N_{\beta}(f)$.
 \end{enumerate}
\end{lemma}
Therefore condition  \eqref{con:B->Lm} holds with $m=1/\beta$ as a consequence of \eqref{iv}, and condition \eqref{con:norm} is satisfied due to \eqref{ii}.
Now, according to Theorem~1 in \citet{Pei93}, we obtain for all $\alpha, \beta \in (0,1/2)$ with $\alpha<\beta$ that $P$ is a bounded linear operator on $\H_{\alpha,\beta}(\X)$ which satisfies condition \eqref{con:sg}.

We now apply \autoref{the:spectral-gap}. Let $s$, $\alpha$, and $\G$ be as in the statement of \autoref{thm:ite}.
By choosing $\beta=(s-1)/s<\frac12$, we have $\alpha<\beta$ and thus \eqref{con:in_B} -- \eqref{con:norm} hold
for the space $\B=\H_{\alpha,\beta}(\X)$ with $m=1/\beta$.
Further, for any $g\in\G$, we have $g\in\H_{\alpha,\beta}(\X)$ and 
$\|g\|_{\H_{\alpha,\beta}}\le \|g\|_{\H_\alpha}$. Therefore, condition \eqref{con:entropy_exp} is also satisfied with respect to the $\H_{\alpha,\beta}(\X)$-norm.
\end{proof}


\subsection{Dynamical Systems with a Spectral Gap}
Let us mention that the result of \autoref{sec:sqclt-mc} can be adapted to deal with dynamical systems using the Perron--Frobenius operator in place of the Markov operator.
Let $(\X,\A)$ be a measurable space and let $T$ be a measurable transformation on $\X$ which preserves a probability measure $\mu$ on $(\X,\A)$. Let $P$ be the associated
Perron--Frobenius operator defined on $L^1(\mu)$ by the equation
\[
\mu(f\cdot Pg)=\mu(f\circ T\cdot g),\quad \forall f\in L^\infty(\mu), g\in L^1(\mu).
\]
We have the following result which can be derived from \autoref{the:spectral-gap} using relativized kernel as in \citet{HenHer01}, Chapter XI.
\begin{theorem}[Sequential empirical CLT for dynamical systems with a spectral gap]\label{the:dynamical-system}
Let $\F$ be a $\|\cdot\|_\infty$-bounded class of functions from $\X$ to  $\R$. Assume that there exist a Banach space $\B$ and $m\in[1,\infty]$ such that the conditions \eqref{con:in_B}, \eqref{con:B->Lm}, \eqref{con:sg}, and \eqref{con:norm} hold with respect to the Perron--Frobenius operator and replacing $\nu$ by $\mu$.
If there exists a $\|\cdot\|_\infty$-bounded subset $\G\subset \B$ such that \eqref{con:entropy_exp} holds for $s=\frac{m}{m-1}$, then the process $(U_n(f,t))_{\F\times [0,1]}$, defined by $U_n(f,t)=\frac{1}{\sqrt{n}}\sum_{i=1}^{[nt]}\left( f\circ T^i - \mu f \right)$, converges in distribution in $\ell^\infty(\F\times [0,1])$ to a centred Gaussian process $K$ with covariance structure given by
\[
\Cov(K(f,t),K(g,u))=\min\{t,u\}\left(
\sum_{k=0}^\infty \Cov(f,g\circ T^k) + \sum_{k=1}^\infty \Cov(f\circ T^k,g)
\right).
\]
\end{theorem}

As a possible application, we can extend the empirical CLT proved by \citet{ColMarSch04} for a class of expanding maps of the interval, to a sequential empirical CLT. In the situation considered in \citet{ColMarSch04}, the spectral gap property can be established on the space of functions of bounded variation.

We consider a piecewise $C^2$ expanding map $T$ of the interval $[0,1]$ which is topologically mixing. We assume that there is a finite partition of $[0,1]$ by intervals such that $T$ is monotone on each interval and further $\inf_{x\in[0,1]}|(T^{n})'(x)|\ge CK^n$ for some $C>0$ and $K>1$.
As noted in \citet{ColMarSch04} (see also \citet{LasYor73}), there is a unique ergodic invariant probability measure $\mu$ such that $d\mu=h(x)d\lambda$, where $\lambda$ is the Lebesgue measure on $[0,1]$. The function $h$ belongs to the Banach algebra $BV$ of functions of bounded variation. By application of \autoref{the:dynamical-system}, we obtain the following result.
\begin{corollary}
Assume that $\frac{1}{h}\ind_{h>0} \in BV$, and let $\F$ be a $\|\cdot\|_\infty$-bounded class of functions such that there exists a subset $\G$ of $BV$ for which \eqref{con:entropy_exp} holds for some $s\ge1$.
Then the process $(U_n(f,t))_{\I\times [0,1]}$, defined by $U_n(f,t)=\frac{1}{\sqrt{n}}\sum_{i=1}^{[nt]}\left( f\circ T^i - \mu f \right)$, converges in distribution in $\ell^\infty(\I\times [0,1])$ to a centred Kiefer process.
\end{corollary}

Note that, in the usual case where $\F=\{\ind_{[0,t]}\mid t\in[0,1]\}$, this result is not new. It can be derived from the result of Section 3 in \citet{DedMerRio13}, since the coefficient $\beta_{2,X}(n)$  which is considered in that paper decreases exponentially fast in our setting (see \citet{DedPri07}, Section 6.3). In \citet{DedMerRio13}, the result is stronger since a strong approximation by a Kiefer process is proved. This implies our weak convergence result.

\begin{proof}
Recall (see \citet{HenHer01}) that the Perron-Frobenius operator $P$ associated with $T$ and $\lambda$ has a spectral gap on $BV$: $1$ is a simple eigenvalue with eigenfunction $h$, and the rest of the spectrum is in a disk of radius strictly smaller than $1$.
Further, the space $BV$ satisfies assumptions \eqref{con:in_B} and \eqref{con:B->Lm} (with $m=+\infty$) and the operator $P$ satisfies 
$\|P^nf- (\lambda f)h \|_{BV} \leq \kappa \|f\|_{BV} \theta^n$ for some $\kappa>0$, $\theta\in[0,1)$, and all $f\in BV$. $BV$ being a Banach algebra, condition \eqref{con:norm} is also satisfied.

In general, the Lebesgue measure is not the invariant measure, i.e. $h$ is not $1$. Thus, we define the set 
$I_h=\{x\in[0,1]\mid h(x)>0\}$ and for functions defined on $I_h$,
 we introduce the operator $P_h$
defined by $P_hf(x)=\frac{1}{h(x)}P(fh)(x)$.
Note that, since $\mu(I_h)=1$, every function $f$ defined on $[0,1]$ is $\mu$ almost surely equal to the function defined by $f$ on $I_h$ and $0$ on $[0,1]\backslash I_h$. With this remark, we can easily check that $\mu(f\cdot P_hg)=\mu(f\circ T\cdot g)$, for all $f\in L^\infty(\mu)$ and $g\in L^1(\mu)$. Then $P_h$ is the Perron-Frobenius operator associated with $T$ and $\mu$.
Now, if a function $f$ is defined on $I_h$, the function $fh$ can be considered on $[0,1]$ by giving the value $0$ on  $[0,1]\backslash I_h$. We introduce the space $\B_h=\{f:I_h\to\R\mid fh\in BV\}$ equipped with the norm $\|f\|_h=\|fh\|_{BV}$. Let us check that the assumptions of \autoref{the:dynamical-system} are satisfied for $P_h$ and $\B_h$.

Clearly, $\B_h$ satisfies the condition \eqref{con:in_B}.
The fact that $\frac{1}{h}\ind_{h>0} \in BV$ gives \eqref{con:B->Lm} (with $m=+\infty$) and \eqref{con:norm}. From the spectral decomposition of $P$ we derive the spectral decomposition of $P_h$ and we obtain the condition \eqref{con:sg} on the space $\B_h$ (with $\mu$ instead of $\nu$).
Thus \autoref{the:dynamical-system} can be applied in this situation.
\end{proof}

As a simple example, we can consider any class of functions $\F=\{f_t\mid t\in[0,1]\}$ indexed by a parameter $t\in[0,1]$ that satisfies:
\begin{itemize}
 \item for all $t\in[0,1]$,  $f_t$ is a non-increasing function bounded by $1$,
 \item for all $0\le t\le u\le1$, $f_t\le f_u$,
 \item the function $t\in[0,1]\mapsto \mu f_t$ is $\alpha$-H\"older for some $\alpha\in(0,1]$.
\end{itemize}
Indeed, in this situation the choice $\G=\F$ is possible. For all $t\in[0,1]$, $f_t$ is $BV$ with $\|f_t\|_{BV}\le 2$. Now, fix $\varepsilon>0$ and choose $m=\lfloor \varepsilon^{-\frac{1}{\alpha}}\rfloor$. 
Let $t_i=\frac{i}{m}$, $i=0,\ldots,m$. For all $t\in[0,1]$, there exists $i\in\{1,\ldots,m\}$ such that $t_i\le t\le t_{i+1}$ and thus $f_{t_i}\le f_t \le f_{t_{i+1}}$. Further, 
$$
\|f_{t_i}-f_{t_{i+1}}\|_1=\mu f_{t_i} - \mu f_{t_{i+1}} \le C (\frac1m)^\alpha \le C\varepsilon.
$$
This shows that $N(\varepsilon,2,\F,\F,\L^1(\mu))= O(\varepsilon^{-\frac{1}{\alpha}})$ as $\varepsilon\to 0$ and gives \eqref{con:entropy_exp}.


\medskip

\citet{Gou09} gave examples of expanding maps of the interval for which the Perron-Frobenius operator does not act on the space of bounded variation functions, but acts on the space of Lipschitz functions with a spectral gap property. These examples also satisfy the assumptions of our theorem and thus sequential empirical CLTs can be proved. Note that the space of Lipschitz functions is a Banach algebra and thus condition \eqref{con:norm} is trivially satisfied. Further, the usual class of the indicator functions of intervals can be well approximated by Lipschitz functions, and the condition \eqref{con:entropy_exp} is verified for this class.


\section{Statistical applications}\label{sec:stat}	

As mentioned in the introduction, sequential empirical CLTs can be applied to derive asymptotic distributions in change-point tests based on the empirical distribution function.  
We shall consider below the natural generalization of the process $T_n$, introduced in \autoref{sec:intro}, to processes taking values in a measurable space $\X$.
Let $(X_i)_{i\in\No}$ be a $\X$-valued stationary process, and $\F$ be a class of function on $\X$. As before, we denote the empirical measure by $F_n(f):=n^{-1} \sum_{i=1}^n f(X_i)$, $n\in\N$,
and we set $F_0(f)=0$.
For $j\in\{1,\ldots,n\}$, we define $F_{j,n}(f):= (n-j+1)^{-1} \sum_{i=j}^n f(X_i)$ and set $F_{n+1,n}(f):=0$. 
Consider the $\ell^\infty(\F\times[0,1])$-valued process $R_n=(R_n(f,t))_{(f,t)\in\F\times[0,1]}$ given by
\[
 R_n(f,t):=\sqrt{n} \frac{[nt]}{n} \frac{n-[nt]}{n}  \bigl(F_{[nt]}(f) - F_{[nt]+1,n}(f)\bigr).
\]
The following theorem gives the asymptotic distribution of $R_n$.

\begin{proposition}\label{the:motivation}
Assume that $(X_i)_{i\in\No}$ satisfies the sequential empirical CLT with indexing class $\F$ and limit process $K$, that is, $U_n \dconv K$ in $\ell^\infty(\F\times[0,1])$ as $n\to\infty$, where $K$ denotes a tight centred Gaussian process.
Then
\[ 
R_n \dconv (K(f,t)-t K(f,1))_{(f,t)\in\F\times[0,1]}
\]
in $\ell^\infty(\F\times[0,1])$ to as $n\to\infty$.
\end{proposition}

\begin{proof} 
Let $\mu$ denote the distribution function of the $X_i$. For $t\in[1/n,1)$ we have 
\begin{align}
&F_{[nt]}(f)-F_{[nt]+1,n}(f) \notag\\
\shoveright&=  \frac{1}{[nt]} \sum_{i=1}^{[nt]} f(X_i)
 -\frac{1}{n-[nt]} \sum_{i=[nt]+1}^n f(X_i) \notag\\
\shoveright&=  \frac{1}{[nt]} \sum_{i=1}^{[nt]} \left( f(X_i) - \mu f\right)
 -\frac{1}{n-[nt]} \sum_{i=[nt]+1}^n \left( f(X_i) - \mu f\right) \notag\\
\shoveright&= \left(\frac{1}{[nt]}+\frac{1}{n-[nt]} \right) \sum_{i=1}^{[nt]} \left( f(X_i) - \mu f\right)
 -\frac{1}{n-[nt]} \sum_{i=1}^n \left( f(X_i) - \mu f\right) \notag\\
\shoveright&= \frac{1}{\sqrt{n}}\frac{n}{[nt]} \frac{n}{n-[nt]}\, U_n(f,t)  
 - \frac{1}{\sqrt{n}} \frac{1}{t} \frac{n}{n-[nt]}\, t U_n(f,1).\label{eq:tm050313-0}
\end{align}
Further, by definition we have $R_n(f,1)=0$ and $R_n(f,t)=0$ for $t\in[0,1/n)$. 
Since also $U_n(f,t)=0$ for $t\in[0,1/n)$, we obtain with \eqref{eq:tm050313-0} that
\begin{align}
 R_n(f,t) 
&= U_n(f,t) - \frac{[nt]}{n}\, U_n(f,1), \notag\\
&= U_n(f,t)- tU_n(f,1) + \frac{nt-[nt]}{n} U_n(f,1) \quad\text{for all}\ t\in[0,1]. \label{eq:tm050313-1}
\end{align}
Let $A_n$ denote the $\F\times[0,1]$-indexed processes given by $A_n(f,t):=\bigl((nt-[nt])/n\bigr) U_n(f,t)$. 
Since $\sup_{t\in[0,1]}|(nt-[nt])/n| \ra 0$ as $n\to\infty$, by Slutsky's Theorem and the sequential empirical CLT, $A_n$ converges in distribution (and thus in probability) to zero. 
Another application of Slutsky's theorem and the sequential empirical CLT on \eqref{eq:tm050313-1} yields
\begin{align*}
R_n = \bigl(U_n(f,t)-tU_n(f,1)\bigr)_{(f,t)\in\F\times [0,1]} + A_n
\dconv \bigl( K(f,t)-t K(f,1) \bigr)_{(f,t) \in \F\times[0,1]}.
\end{align*}
Here we have applied the continuous mapping theorem in the final step.
\end{proof}

\begin{remark}
Note that, in the setting of \autoref{the:spectral-gap} and \autoref{the:dynamical-system}, the process $K$ is a Kiefer process (that is the covariance structure is given by \eqref{eq:cov-structure}).
\end{remark}

An application of the continuous mapping theorem with the supremum-functional to the above theorem yields the following proposition about the asymptotic distribution of the test statistic $T_n$ defined by
\[
T_n:= \max_{0 \leq k\leq n} \sup_{f\in\F} \frac{k}{n} \Bigl( 1-\frac{k}{n} \Bigr) \sqrt{n}
  \bigl| F_k(f)-F_{k+1,n}(f)  \bigr|.
\]
\begin{theorem}\label{pro:test-statistic}
If $(X_i)_{i\in\N}$ satisfies the sequential empirical CLT, then under the null hypothesis $\mathbf{H}_0$ we have the convergence 
\[
T_n \dconv
	\sup_{f\in \F,\ t\in[0,1]} |K(f,t)-t K(f,1)|.
\]
\end{theorem}

\begin{proof}
$R_n(f,\cdot)$ is obviously constant on the intervals $\bigl[k/n,(k+1)/n\bigr)$, $k=0,\ldots,n-1$ and further 
$ R_n\bigl(f,{k}/{n}\bigr) = {k}/{n} ( 1-{k}/{n} ) \sqrt{n} \bigl( F_k(f)-F_{k+1,n}(f)  \bigr)$
for $k=0,\ldots,n$. Thus $T_n=\sup_{f\in\F, t\in[0,1]} R_n(f,t)$ and we can apply the continuous mapping theorem with
\[
\ell^\infty(\F\times[0,1]) \lra \R,\quad \phi \mapsto \sup_{f\in\F,\ t\in[0,1]} |\phi(f,t)|. 
\]
\end{proof}


\section[Proof of \autoref*{the:sq-ep-clt_mm}]{Proof of \autoref{the:sq-ep-clt_mm}}\label{sec:proof_SECLT}

As proved in \citet{DehDur11}, multiple mixing processes satisfy the following $2p$-th moment bound.
\begin{assumption}[Moment bounds for $\C$-observables]	\label{asp:mbound}
There exist $p\in\N$, $s\ge 1$, and monotone increasing functions $\Phi_1,\ldots,\Phi_p:\R_+\lra\R_+$, 
 \begin{equation}\label{eq:2pbound}
\E\left[\left(\sum_{i=1}^{n}(f(X_i)-\mu f)\right)^{2p}\right]\leq
\sum_{i=1}^p n^i\|f\|_{s}^i \Phi_i (\|f\|_{\C}) \quad \text{for all}\ f\in\C \text{ with } \|f\|_\infty\le 1.
 \end{equation}
\end{assumption}

\medskip

We shall obtain \autoref{the:sq-ep-clt_mm} as a consequence of the more general following result.
\begin{theorem}\label{the:s-ep-clt}
Let $(X_i)_{i\in\No}$ be an $\X$-valued stationary process with marginal distribution $\mu$ and let $\F$ be a uniformly bounded class of measurable functions on $\X$.
Suppose that for some normed vector space $\C$ of measurable functions on $\X$, \autoref{asp:clt_seq} and \autoref{asp:mbound} hold. 
Moreover, assume that there exist a subset $\G$ of $\C$ which is bounded in $\|\cdot\|_\infty$-norm, a constant $r>-1$ and a monotone increasing function $\Psi:\R_+\lra\R_+$ such that 
			\begin{equation}
				\label{eq:bracket}
				\int_0^1 \varepsilon^{r}\sup_{\epsilon \le \delta \le 1} N^2\bigl(\delta,
				\Psi\left(\delta^{-1}\right),\F,\G,\L^s(\mu)\bigr)
				d\varepsilon <\infty.
			\end{equation}
If			
	\begin{align}
		\Phi_i(2\Psi(x)) = O(x^{\gamma_i}), \label{eq:Phi_of_Psi}
	\end{align}
for some non-negative constants $\gamma_i$ 
such that
	\begin{align}
		\gamma_i < 2p - (i + r + 2) \label{eq:gamma_i-p_seq},
	\end{align}
then the sequential empirical process $U_n$ converges in distribution in 
$\ell^\infty(\F\times[0,1])$ to a tight Gaussian process $K$.
\end{theorem}

\begin{proof}[Proof of \autoref{the:sq-ep-clt_mm}]
Under multiple mixing (\autoref{asp:mm}), \autoref{asp:mbound} holds for all $p\ge1$ and we can specify that $\Phi_i(x)=c \log^{2p+(d_0-1)i}(x+1)$ for some $c>0$ depending only on $p$, see \citet{DehDur11}.
Observe that, choosing $\Psi:= \exp(C\id^{{1}/{\gamma}})$ for some $C>0$ and $\gamma>1$, we have $\Phi_i(2 \Psi(x)) = O(x^{(2p+(d_0-1)i)/{\gamma}})$. 
Therefore, the conditions \eqref{eq:Phi_of_Psi} and \eqref{eq:gamma_i-p_seq} hold for sufficiently large $p\in\N$ as soon as $\gamma>d_0+1$. 
With this choice of $\Psi$, the condition \eqref{eq:bracket} is exactly the condition \eqref{con:entropy_exp}. Thus \autoref{the:sq-ep-clt_mm} is a consequence of \autoref{the:s-ep-clt}.
\end{proof}

\medskip

The proof of \autoref{the:s-ep-clt} extends the idea introduced in \citet*{DehDurTus14a}, taking into account the time parameter due to the sequential case.
The main idea is to introduce some approximation $U_n^{(q)}$ for the original process $U_n$, which is based on functions in $\G$ and thus can be controlled by \autoref{asp:clt_seq} and \ref{asp:mbound}. 
The approximation can be constructed as follows:
For all $q\ge 1$, there exist two sets of $N_q:=N(2^{-q},\Psi(2^q),\F,\G,\L^s(\mu))$ functions $\{g_{q,1},\dots,g_{q,N_q}\}\subset\G$ and $\{g_{q,1}',\dots,g_{q,N_q}'\}\subset\G$, 
such that
\begin{align}
	\|g_{q,i}-g_{q,i}'\|_s&\le2^{-q},& \|g_{q,i}\|_\C&\le \Psi(2^{q}),& \|g_{q,i}'\|_\C&\le \Psi(2^q) \label{eq:tmp150213.3}
\end{align} 
and for all $f\in\F$, there exists some $i$ such that $g_{q,i}\le f\le g_{q,i}'$.
Further, by \eqref{eq:bracket}, 
\begin{align}
	\sum_{q\ge 1}2^{-(r+1)q} N_q^2<\infty.\label{eq-bracket_sum}
\end{align}
To approximate the indexing function $f\in\F$, construct a partition of $\F$ into $N_q$ subsets $\F_{q,i}$ such that for each $f\in\F_{q,i}$ one has $g_{q,i}\leq f \leq g_{q,i}'$. 
We use the notation $\pi_q f=g_{q,i^*}$ and $\pi_q' f=g_{q,i^*}'$, where $i^*$ is the uniquely defined integer such that $f\in\F_{q,i^*}$.
To approximate the time parameter we use the partition of $[0,1]$ into subsets $\mathcal{T}_{q,j}$, $j=1\ldots,2^q$, given by $\mathcal{T}_{q,j}:=[(j-1) 2^{-q},j 2^{-q})$ for $j<2^q$ and $\mathcal{T}_{q,2^q}:=[1-2^{-q},1]$.
For $t\in[0,1]$ we define
$\tau_q t:= \max\{(j-1) 2^{-q} \leq t : {j=1,\ldots,2^q}\}$ 
and further $\tau_q' t:=\tau_qt + 2^{-q}$. 
We extend the notation introduced in \autoref{sec:def} to arbitrary $\mu$-integrable functions $f:\X\lra\R$ by setting
\[ 
F_n(f):= \frac{1}{n} \sum_{i=1}^{n} f(X_i) 
\]
and for $t\in[0,1]$
\[  
U_n(f,t)
:= \frac{[nt]}{\sqrt{n}} \bigl( F_{[nt]}(f)- \mu(f) \bigr)
= \frac{1}{\sqrt{n}} \sum_{i=1}^{[nt]} \bigl( f(X_i)-\mu(f) \bigr).
\]
For each $q\ge 1$, we introduce the approximating process
\[ 
U_n^{(q)}(f,t):=U_n(\pi_qf,\tau_qt)= \frac{1}{\sqrt{n}} \sum_{i=1}^{[n\tau_qt]} \left( \pi_q f(X_i)-\mu(\pi_qf) \right).
\]
Note that this process
is
constant on each $\F_{q,i}\times\mathcal{T}_{q,j}$.

The approximating processes $U_n^{(q)}$ will help us to establish the weak convergence of the process $U_n$.
Using Theorem 2.1 in \citet*{DehDurTus14a}, we see that it is sufficient to show that there exist processes $ U^{(q)}\in\ell^{\infty}(\F\times[0,1])$, $q\ge 1$, such that
\begin{align}
 &U_n^{(q)} \dconv U^{(q)}
		\quad \text{as}\ n\rightarrow \infty\ \text{for all}\ q\geq 1,
\label{vor:xm_conv}
\end{align}
and
\begin{align}
 &\limsup_{n \rightarrow \infty} \P^*\bigl( \| U_n - U_n^{(q)} \|_{\infty} \geq 
\delta\bigr) \longrightarrow 0
		\quad \text{as}\ q \rightarrow \infty\ \text{for all}\ \delta>0.
\label{vor:xn_near_xnm} 
\end{align}

We will establish the conditions \eqref{vor:xm_conv} and \eqref{vor:xn_near_xnm} in the two following propositions:

\begin{proposition}\label{pro:u_n^q-u^q}
If \autoref{asp:clt_seq} holds, then for all $q\in\N$ the process $(U_n^{(q)}(f,t))_{(f,t)\in\mathcal{F}\times[0,1]}$ converges in distribution to a piecewise constant Gaussian process $(U^{(q)}(f,t))_{(f,t)\in\mathcal{F}\times[0,1]}$ as $n\to\infty$. 
\end{proposition}

\begin{proposition}\label{pro:u_n^q-u_n} 
Assume that \autoref{asp:mbound} holds for some $p\in\N$, $s\geq 1$ and some monotone increasing functions $\Phi_1,\ldots,\Phi_p:\R_+\lra\R_+$.
Moreover, suppose there exists a constant $r>-1$ and a monotone increasing function $\Psi:\R_+\lra \R_+$ such that \eqref{eq:bracket} holds.
If \eqref{eq:Phi_of_Psi} holds for some non-negative constants $\gamma_i$ satisfying \eqref{eq:gamma_i-p_seq}, then for all $\varepsilon,\eta>0$ there exists some $q_0$ such that for all $q\geq q_0$
\[
 \limsup_{n\rightarrow\infty}\P^*\left(\sup_{t\in[0,1]} \sup_{f\in\F}\left|U_n(f,t)-U_n^{(q)}(f,t)\right|>\varepsilon\right)\le\eta.
\]
\end{proposition}

\begin{proof}[Proof of \autoref{the:s-ep-clt}]
By \autoref{pro:u_n^q-u^q} the convergence \eqref{vor:xm_conv} holds, while \eqref{vor:xn_near_xnm} is satisfied due to \autoref{pro:u_n^q-u_n}. Therefore, by Theorem 2.1 in \citet*{DehDurTus14a},  $U_n$ converges in distribution to
an $\ell^{\infty}(\F\times[0,1])$-valued, separable random variable $K$. Furthermore, we know that $U^{(q)}$ is a piecewise constant Gaussian process which converges in distribution to $K$. Thus $K$ is Gaussian, too. Since $\ell^\infty(\F\times[0,1])$ is complete, the tightness of $K$ follows from the separability (c.f. Lemma 1.3.2 in \citet{VanWel96}).
\end{proof}


\begin{proof}[Proof of \autoref{pro:u_n^q-u^q}]
Since by construction $\pi_q f\in\G$ for all $f\in\F$,
due to \autoref{asp:clt_seq}, 
the finite dimensional process $(U_n^{(q)}(f_1,t_1),\ldots,U_n^{(q)}(f_k,t_k)
)$ converges in distribution to some multi-dimensional normal distributed random variable $(U^{(q)}(f_1,t_1),\ldots,U^{(q)}(f_k,t_k))$ 
for all fixed $k\in\N$, $f_1,\ldots,f_k\in\F$, $t_1,\ldots,t_k\in[0,1]$.
All $U_n^{(q)}$, $n\in\N$, are constant on each $\F_{q,i}\times\mathcal{T}_{q,j}$, $i=1,\ldots,N^q$, $j=1,\ldots,2^q$. Therefore $U^{(q)}$ is constant on all $\F_{q,i}\times\mathcal{T}_{q,j}$, too. 
Since these sets form a partition of $\F\times[0,1]$, the finite dimensional convergence yields the convergence in distribution of the whole process $(U_n^{(q)}(f,t))_{(f,t)\in\mathcal{F}\times[0,1]}$.
\end{proof}


\begin{proof}[Proof of \autoref{pro:u_n^q-u_n}]
Let $\overline{Z}:=Z-\E Z$ denote the centring of a random variable $Z$ and
observe that for any random variables $Y_l \leq Y \leq Y_u$ the inequality 
\[ |\overline{Y}-\overline{Y_l}|\leq |\overline{Y_u}-\overline{Y_l}| + \E|{Y_u}-{Y_l}| \]
holds. 
Since for $f\in\F$, $k\in\No$ we have
$F_{[nt]}(\pi_{q+k} f,t)\le F_{[nt]}(f,t)\le F_{[nt]}(\pi_{q+k}'f, t)$, using that $\|\cdot\|_1\leq\|\cdot\|_s$ for $s\ge 1$ and applying \eqref{eq:tmp150213.3}, we obtain
\begin{align}
  &\bigl| U_n(f,t) - U_n(\pi_{q+k} f,t) \bigr| \notag\\
	\shoveright&\leq \bigl| U_n(\pi_{q+k}'f,t) - U_n(\pi_{q+k} f,t) \bigr| + \frac{[nt]}{\sqrt{n}}\, \E\bigl| F_{[nt]}(\pi_{q+k}' f- \pi_{q+k} f) \bigr| \notag\\
	\shoveright&\leq \bigl| U_n(\pi_{q+k}'f,t) - U_n(\pi_{q+k} f,t) \bigr| + \sqrt{n} 2^{-(q+k)}. \label{eq:tmp-020412-1}
\end{align}
Moreover, for all $n\geq2^{q+k}$ and $g\in\G$
\begin{align}
	\bigl| U_n(g,t)-U_n(g,\tau_{q+k}t) \bigr|
	&= \frac{1}{\sqrt{n}} \Biggl| \sum_{i=[n\tau_{{q+k}t}]+1}^{[nt]} g(X_i)- \mu(g) \Biggr| \notag\\
	&\leq 2M n^{-\frac{1}{2}} ([nt]-[n\tau_{q+k}t]) \notag\\
	&\leq 4M\sqrt{n} 2^{-(q+k)}, \label{eq:tmp-020412-2}
\end{align}
where $M := \sup \{ \|g\|_\infty : g\in\G \}$ is finite by assumption.
Analogously to the processes $U_n^{(q)}$, we introduce the processes $U_n'^{(q)}$ given by
\[
U_n'^{(q)}(f,t):=U_n(\pi_q'f,\tau_q't).
\]
An application of the triangle inequality, \eqref{eq:tmp-020412-1}, and \eqref{eq:tmp-020412-2} yields
\begin{align}
 \left| U_n(f,t) - U_n^{({q+k})}(f,t) \right|
 \leq \left| U_n'^{({q+k})}(f,t) - U_n^{({q+k})}(f,t) \right| + (4M+1) \sqrt{n}2^{-{q+k}}. \label{tmp-020412-4}
\end{align}
Combining \eqref{tmp-020412-4} with a telescopic sum argument, one obtains for any $K\ge 1$
\begin{align}
 &\left| U_n(f,t) - U_n^{(q)}(f,t) \right| 
\notag\\
 \shoveright&= \Biggl| \biggl\{ \sum_{k=1}^K U_n^{(q+k)}(f,t)-U_n^{(q+k-1)}(f,t) \biggr\} + U_n(f,t) - U_n^{(q+K)}(f,t)\Biggr| 
\notag\\
 \shoveright&\leq \biggl\{ \sum_{k=1}^K \left| U_n^{(q+k)}(f,t)-U_n^{(q+k-1)}(f,t)\right| \biggr\}+ \left| U_n'^{(q+K)}(f,t) - U_n^{(q+K)}(f,t)\right| 
\notag\\ 
\shoveright&\hspace{3ex}
+ (4M+1)\sqrt{n}2^{-(q+K)}. \label{eq:tmp050313-0} 
\end{align}
To assure $\epsilon/4 \leq (4M+1) \sqrt{n} 2^{-(q+K)} \leq \epsilon/2$, choose $K=K_{n,q}$, given by
\[
K_{n,q}:= \left[\log_2 \left( \frac{4(4M+1)\sqrt{n}}{2^{q}\epsilon} \right)\right].
\]
For each $i=1,\ldots,N_q$, $j=1,\ldots,2^q$, inequality \eqref{eq:tmp050313-0} implies
\begin{align*}
	\sup_{t\in\mathcal{T}_{q,j}} \sup_{f\in\F_{q,i}} |U_n(f,t) - U_n^{(q)}(f,t)| 
	\leq& \biggl\{ \sum_{k=1}^{K_{n,q}} \sup_{t\in\mathcal{T}_{q,j}} \sup_{f\in\F_{q,i}} \left| U_n^{(q+k)}(f,t)-U_n^{(q+k-1)}(f,t)\right|\biggr\} \notag\\ 
	& +\sup_{t\in\mathcal{T}_{q,j}} \sup_{f\in\F_{q,i}} \left| U_n'^{(q+K)}(f,t) - U_n^{(q+K)}(f,t)\right| 
	+ \frac{\epsilon}{2}.
\end{align*}
Set $\epsilon_k= \epsilon /(4k(k+1))$. 
Then $\sum_{i=1}^{\infty}\epsilon_k=\epsilon/4$ and for all $i=1,\ldots,N_q$ we have
\begin{align}
	&\P^*\biggl(\sup_{t\in\mathcal{T}_{q,j}} \sup_{f\in\F_{q,i}} |U_n(f,t) - U_n^{(q)}(f,t)| \geq \epsilon \biggr) \notag\\ 
	\shoveright&\leq \Biggl\{ \sum_{k=1}^{K_{n,q}} \P^*\biggl(  \sup_{t\in\mathcal{T}_{q,j}} \sup_{f\in\F_{q,i}} \left| U_n^{(q+k)}(f,t)-U_n^{(q+k-1)}(f,t)\right| \geq \epsilon_k\biggr) \Biggr\} \notag\\ 
	\shoveright& \hspace{3ex}+ \P^* \biggl( \sup_{t\in\mathcal{T}_{q,j}} \sup_{f\in\F_{q,i}} \left| U_n'^{(q+K)}(f,t) - U_n^{(q+K)}(f,t)\right| \geq \frac{\epsilon}{4}\biggr). \label{eq:tmp050313-1}
\end{align}
Recall that $(\pi_{q+k},\tau_{q+k})$ and thus $U_n^{(q+k)}$ and $U_n'^{(q+k)}$ are constant on each $\F_{q+k,i}\times\mathcal{T}_{q+k,j}$%
, $i=1,\ldots N_{q+k}$, $j=1,\ldots,2^{q+k}$,
and thus the suprema on the r.h.s.\ of inequality \eqref{eq:tmp050313-1} are in fact maxima over finite numbers of functions.
Therefore the outer probabilities may be replaced by usual probabilities here. 
Now, for each $k\in\N$, choose a set $\F(k)$ of at most $N_{k-1}N_k$ functions in $\F$, such that $\F(k)$ contains at least one function in each non empty $\F_{k,i} \cap \F_{k-1,i'}$, $i=1,\ldots,N_{k}$, $i'=1,\ldots,N_{k-1}$. 
For $q\in\N$ and $i\in\{1,\ldots,N_q\}$, define 
\begin{align*} 
F_{k,q,i}&:= \F_{q,i}\cap \F(q+k)\\
T_{k,q,j}&:=\bigl\{ (j-1) 2^{-q} + (m-1) 2^{-(q+k)} : m\in\{1,\ldots,2^{k}\} \bigr\}.
\end{align*}
Inequality \eqref{eq:tmp050313-1} implies
\begin{align}
	&\P^*\biggl(\sup_{t\in\mathcal{T}_{q,j}} \sup_{f\in\F_{q,i}} |U_n(f,t) - U_n^{(q)}(f,t)| \geq \epsilon \biggr) 
	\notag\\
	\shoveright&\leq \Biggl\{ 
	\sum_{k=1}^{K_{n,q}} \sum_{t\in T_{k,q,j}} \sum_{f\in F_{k,q,i}} \!\!\!
	\P\Bigl( \Bigl| U_n^{(q+k)}(f,t)-U_n^{(q+k-1)}(f,t)\Bigr| \geq \epsilon_k\Bigr)\Biggr\} 
	\notag\\ 
	\shoveright& \hspace{3ex}
	+ \sum_{t\in T_{K_{n,q},q,j}} \sum_{f\in F_{K_{n,q},q,i}} \!\!\!
	\P \Bigl( \Bigl| U_n'^{(q+K_{n,q})}(f,t) - U_n^{(q+K_{n,q})}(f,t)\Bigr| \geq \frac{\epsilon}{4}\Bigr) 
	\notag
	\\
%
%
	\shoveright&\leq \Biggl\{
	\sum_{k=1}^{K_{n,q}} \sum_{t\in T_{k,q,j}} \sum_{f\in F_{k,q,i}} \!\!\!
	\P\Bigl( \Bigl| U_n(\pi_{q+k}f,\tau_{q+k-1}t)-U_n(\pi_{q+k-1}f,\tau_{q+k-1}t)\Bigr| \geq \frac{\epsilon_k}{2}\Bigr) 
	\notag\\
	\shoveright&\hspace{6ex}
	+	\P\Bigl(  \Bigl| U_n(\pi_{q+k}f,\tau_{q+k}t)-U_n(\pi_{q+k}f,\tau_{q+k-1}t)\Bigr| \geq \frac{\epsilon_k}{2}\Bigr) 
	\Biggr\}
	\notag\\
	\shoveright&\hspace{3ex}
	+	\sum_{t\in T_{K_{n,q},q,j}} \sum_{f\in F_{K_{n,q},q,i}} \!\!\!
	\P \Bigl( \Bigl| U_n(\pi'_{q+K_{n,q}}f,\tau_{q+K_{n,q}}t) 
	- U_n(\pi_{q+K_{n,q}}f,\tau_{q+K_{n,q}}t)\Bigr| \geq \frac{\epsilon}{8}\Bigr) 
	\notag\\ 
	\shoveright&\hspace{6ex} 
 	+ \P \Bigl( \Bigl| U_n(\pi'_{q+K_{n,q}}f,\tau'_{q+K_{n,q}}t) - U_n(\pi'_{q+K_{n,q}}f,\tau_{q+K_{n,q}}t)\Bigr| \geq \frac{\epsilon}{8}\Bigr).
	\notag
\end{align}
Applying Markov's inequality on the $2p$-th moments, we obtain
\begin{align}
	&\P^*\biggl(\sup_{t\in\mathcal{T}_{q,j}} \sup_{f\in\F_{q,i}} |U_n(f,t) - U_n^{(q)}(f,t)| \geq \epsilon \biggr) 
	\notag\\
		\shoveright&\leq \Biggl\{
	\sum_{k=1}^{K_{n,q}} \sum_{t\in T_{k,q,j}} \sum_{f\in F_{k,q,i}} 
	{\Bigl(\frac{\epsilon_k}{2}\Bigr)}^{-2p} \Bigl( 
	\E \bigl| U_n(\pi_{q+k}f,\tau_{q+k-1}t)-U_n(\pi_{q+k-1}f,\tau_{q+k-1}t)\bigr|^{2p}
	\notag\\
	\shoveright&\hspace{6ex}
	+	\E\bigl| U_n(\pi_{q+k}f,\tau_{q+k}t)-U_n(\pi_{q+k}f,\tau_{q+k-1}t)\bigr|^{2p} 
			\Bigr)
		\Biggr\}
	\notag\\
	\shoveright&\hspace{3ex}
	+	\sum_{t\in T_{K_{n,q},q,j}} \sum_{f\in F_{K_{n,q},q,i}} \!\!
	{\Bigl(\frac{\epsilon}{8}\Bigr)}^{-2p} \Bigl( 
	\E \bigl| U_n(\pi'_{q+K_{n,q}}f,\tau_{q+K_{n,q}}t) 
	- U_n(\pi_{q+K_{n,q}}f,\tau_{q+K_{n,q}}t)\bigr|^{2p} 
	\notag\\ 
	\shoveright&\hspace{6ex} 
 	+ \E \bigl| U_n(\pi'_{q+K_{n,q}}f,\tau'_{q+K_{n,q}}t) - U_n(\pi'_{q+K_{n,q}}f,\tau_{q+K_{n,q}}t)\bigr|^{2p}
		\Bigr). 
	\label{eq:tmp060313-0}	
\end{align}

We will treat the expected values on the r.h.s.\ of inequality \eqref{eq:tmp060313-0} separately now by using \autoref{asp:mbound} and properties of our brackets used to cover $\F$. 
Recall that by \eqref{eq:tmp150213.3} we have
\begin{align}
 \|\pi_{q+k}f-\pi_{q+k-1}f\|_s &\le  \|\pi_{q+k}f-f\|_s+\|\pi_{q+k-1}f-f\|_s\le 3 \cdot 2^{-(q+k)} \label{eq:tmp150213.1}\\
 \|\pi_{q+k}f-\pi_{q+k}'f\|_s &\le 2^{-(q+k)} \notag\\
 \|\pi_{q+k}f-\pi_{q+k-1}f\|_\C &\le 2 \Psi(2^{q+k}) \label{eq:tmp150213.2}\\
 \| \pi_{q+k}f-\pi_{q+k}'f\|_\C &\le 2 \Psi(2^{q+k}).\notag
\end{align}
For convenience, throughout the rest of the proof will write $x \ll y$ if there is some finite constant $C\in(0,\infty)$ such that $x \leq C y$, where $C$ may only depend on global parameters of the corresponding statement.
Applying successively \eqref{eq:2pbound}, \eqref{eq:tmp150213.1}, \eqref{eq:tmp150213.2}, and \eqref{eq:Phi_of_Psi} we have
\begin{align}
	&\E \bigl| U_n(\pi_{q+k}f,\tau_{q+k-1}t)-U_n(\pi_{q+k-1}f,\tau_{q+k-1}t)\bigr|^{2p} 
	\notag\\
	\shoveright&\ll n^{-p} \sum_{\ell=1}^p n^\ell \| \pi_{q+k}f - \pi_{q+k-1}f \|_s^\ell \Phi_\ell (\| \pi_{q+k}f - \pi_{q+k-1}f\|_\C) 
	\notag\\ 
	\shoveright&\ll \sum_{\ell=1}^p n^{-(p-\ell)} 2^{(\gamma_\ell-\ell)(q+k)} \label{eq:tmp060313-1}
\end{align}
and analogously
\begin{align}
	\E \bigl| U_n(\pi'_{q+K_{n,q}}f,\tau_{q+K_{n,q}}t) - U_n(\pi_{q+K_{n,q}}f,\tau_{q+K_{n,q}}t)\bigr|^{2p}
	 \ll \sum_{\ell=1}^p n^{-(p-\ell)} 2^{(\gamma_\ell-\ell)(q+K_{n,q})}. \label{eq:tmp060313-2}
\end{align}
For fixed $g\in\G$ we have by stationarity
\begin{align}
	\E \bigl|U_n(g,\tau_{q+k}t) -U_n(g, \tau_{q+k-1}t)\bigr|^{2p}
	= n^{-p} \E \left[ \Biggl( \sum_{i=1}^{[n\tau_{q+k}t]-[n\tau_{q+k-1}t]} \bigl( g(X_i) - \mu g \bigr) \Biggr)^{2p} \right], \label{eq:tmp060313-3}
\end{align}
where we consider $\sum_{i=1}^0\ldots=0$. 
Note that by construction $\tau_{q+k}t-\tau_{q+k-1}t\in\{0,2^{-(q+k)}\}$ for every $t\in[0,1]$ and therefore
\[
{[n\tau_{q+k}t]-[n\tau_{q+k-1}t]} \leq n2^{-(q+k)}+1 \quad\text{for all}\ n\geq2^{q+k}.
\]
Applying \eqref{eq:2pbound}, \eqref{eq:tmp150213.3}, and \eqref{eq:Phi_of_Psi} to \eqref{eq:tmp060313-3} we obtain
\begin{align}
	\E \bigl|U_n(\pi_{q+k}f,\tau_{q+k}t) -U_n(\pi_{q+k}f, \tau_{q+k-1}t)\bigr|^{2p} 
	&\ll n^{-p} \sum_{\ell=1}^p \bigl(n2^{-(q+k)}\bigr)^\ell\|\pi_{q+k}f\|_{s}^\ell \Phi_\ell (\|\pi_{q+k}f\|_\C) \notag\\
	&\ll \sum_{\ell=1}^p n^{-(p-\ell)} 2^{(\gamma_\ell-\ell)(q+k)} \label{eq:tmp060313-4}
\end{align}
and analogously 
\begin{align}
	\E \bigl| U_n(\pi'_{q+K_{n,q}}f,\tau'_{q+K_{n,q}}t) - U_n(\pi'_{q+K_{n,q}}f,\tau_{q+K_{n,q}}t)\bigr|^{2p}
	&\ll \sum_{\ell=1}^p n^{-(p-\ell)} 2^{(\gamma_\ell-\ell)(q+K_{n,q})}. \label{eq:tmp060313-5}
\end{align}

Now, apply \eqref{eq:tmp060313-1}, \eqref{eq:tmp060313-2}, \eqref{eq:tmp060313-4}, and \eqref{eq:tmp060313-5} to \eqref{eq:tmp060313-0}. We infer
\begin{align}
		&\P^*\biggl(\sup_{t\in\mathcal{T}_{q,j}} \sup_{f\in\F_{q,i}} \Bigl|U_n(f,t) - U_n^{(q)}(f,t)\Bigr| \geq \epsilon \biggr)\notag\\
		\shoveright&\ll  \sum_{k=1}^{K_{n,q}} \#T_{k,q,j} \, \#F_{k,q,i} \, \frac{(k(k+1))^{2p}}{\epsilon^{2p}} \sum_{\ell=1}^p n^{-(p-\ell)} 2^{(\gamma_\ell-\ell)(q+k)}. \label{eq:tmp060313-6}
\end{align}
Recall that by construction of the partitions of $\F$ and $[0,1]$ at the beginning of this section, we have $\sum_{j=1}^{2^q} \# T_{k,q,j} = 2^{q+k}$ and $\sum_{i=1}^{N_q} \# F_{k,q,i} = \#\F(q+k) \leq N_{q+k-1} N_{q+k}$.
Therefore \eqref{eq:tmp060313-6} yields 
\begin{align}
	&\P^*\left(\sup_{t\in[0,1]} \sup_{f\in\F}\left|U_n(f,t)-U_n^{(q)}(f,t)\right|>\varepsilon\right) 
	\notag\\
	\shoveright&\ll \sum_{\ell=1}^p \sum_{k=1}^{K_{n,q}}
	\sum_{j=1}^{2^q} \#T_{k,q,j}
	\sum_{i=1}^{N_q} \#F_{k,q,i}
	k^{4p} 
	n^{-(p-\ell)} 2^{(\gamma_\ell-\ell)(q+k)}  
	\notag\\
	\shoveright&\ll \sum_{\ell=1}^p \sum_{k=1}^{K_{n,q}}
	N_{q+k-1} N_{q+k}
	k^{4p} n^{-(p-\ell)} 2^{(\gamma_\ell-\ell+1)(q+k)}. 
	\notag
\end{align}
This implies that for any $\eta>0$
\begin{align}
	&\P^*\left(\sup_{t\in[0,1]} \sup_{f\in\F}\left|U_n(f,t)-U_n^{(q)}(f,t)\right|>\varepsilon\right) 
	\notag\\
	\shoveright&\ll \sum_{\ell=1}^p 
	n^{-(p-\ell)} \max\left\{ 1 \;,\; 2^{(\gamma_\ell-\ell+r+2+\eta)(q+K_{n,q})} \right\}
	\sum_{k=1}^{K_{n,q}} N_{q+k-1} N_{q+k} k^{4p} 2^{-(r+1+\eta)(q+k)} \notag\\
	\shoveright&\ll \max\left\{ 1 \;,\; \max_{\ell=1,\ldots,p} n^{\frac{1}{2}(\gamma_\ell+\ell-2p+r+2+\eta)} \right\}
	\sum_{k=q+1}^{\infty} N_{k-1} N_{k} k^{4p} 2^{-(r+1+\eta)k}. 
	\label{eq:tmp030312.1}
\end{align}
By \eqref{eq:gamma_i-p_seq} we can choose $\eta$ small enough to assure $\gamma_\ell+\ell-2p+r+2+\eta<0$ for all $\ell=1,\ldots,p$. Thus the factor in front of the sum is uniformly bounded w.r.t.\ $n$. Using \eqref{eq-bracket_sum}, we obtain
\[
\sum^\infty_{k=1} N_{k-1} N_{k} k^{4p}2^{-(r+1+\eta)k} 
\leq \sum^\infty_{k=1} 2^{-(r+1)k} N_{k-1}^2 \cdot k^{4p}2^{-\eta k} + 
\sum^\infty_{k=1} 2^{-(r+1)k}N_{k}^2 \cdot k^{4p}2^{-\eta k} < \infty
\]
for sufficiently small $\eta>0$ which implies that the series in \eqref{eq:tmp030312.1} goes to zero as $q\to\infty$.
\end{proof}

\section[Proof of \autoref*{lem:covariance_seq}]{Proof of \autoref{lem:covariance_seq}}\label{sec:proof_covariance}

In order to simplify the expressions, set $\Psi(x)= \exp(C x^{{1}/{\gamma}})$, where $C$ and $\gamma$ are given by \autoref{the:sq-ep-clt_mm}. Choose $b\in(1,\gamma)$ and observe that,
       \begin{equation}\label{eq:sum_psi_theta}
	\sum_{k=1}^\infty \Psi(k^b) \theta^k < \infty.
	\end{equation}
For $f\in\F$, recall the definition of the approximating functions $\pi_q f$ from \autoref{sec:proof_SECLT} and note that, as a consequence of the entropy condition in \autoref{the:sq-ep-clt_mm}, we know that for every $q\in\N$, 
		\begin{align}
			\| f - \pi_q f\|_s &\leq 2^{-q} \label{eq:pms}\\
			\| \pi_q f \|_\C &\leq \Psi(2^q)\label{eq:pmB},
		\end{align}
where $s\ge 1$ is given in the assumptions of \autoref{the:sq-ep-clt_mm}.
Similarly, for all $g\in\F$ and $k\in\N$ there exist some $g_k\in\G$ satisfying
		\begin{align}
			\|g_k - g\|_s & \leq k^{-b} \label{eq:pks}\\
			\| g_k \|_\C &\leq \Psi(k^b). \label{eq:pkB}
		\end{align}
Let $U^{(q)}$ denote the limit process given in \autoref{pro:u_n^q-u^q}.
Condition \eqref{con:clt_cov} implies that for all $f,g\in\F$, $t,u\in[0,1]$ and $q\in\N$
 	\begin{align*} 
 		&\Cov\bigl(U^{(q)}(f,t) , U^{(q)}(g,u)\bigr)
		\\
 		\shoveright&= \min\{t,u\}\, \biggl\{ 
 		\sum_{k=0}^{\infty} \Cov\bigl( \pi_q f(X_0),  \pi_q g(X_k) \bigr) 
 		+ \sum_{k=1}^{\infty} \Cov\bigl( \pi_q g(X_0) , \pi_q f(X_k) \bigr) \biggr\}. 
 	\end{align*}
Since the auto-covariance functions of a converging Gaussian process converge to the auto-covariance functions of the limit process, the covariance structure of the limit process $K$ of $U^{(q)}$ is given by $\Cov(K(f,t),K(g,u)) =\lim_{q\to\infty} \Cov(U^{(q)}(f,t) , U^{(q)}(g,u))$. Thus it suffices to show that
\begin{align}
 	& \Bigl| \sum_{k=0}^{\infty} \Cov\bigl(\pi_q f(X_0) , \pi_q g(X_k)\bigr) - \Cov\bigl(f(X_0) , g(X_k)\bigr) \Bigr| \label{eq:tmp1501-1}\\
 	\shoveright&+ \Bigl| \sum_{k=1}^{\infty} \Cov\bigl(\pi_q g(X_0) , \pi_q f(X_k)\bigr) - \Cov\bigl(g(X_0) , f(X_k)\bigr) \Bigr| \notag
 	\lra 0\quad \text{as}\ q\to\infty.
\end{align} 	
By symmetry, both series can be treated the same way.
Let $k(q):= 2^{{q}/{b}}$.
We consider the series in line \eqref{eq:tmp1501-1}. We have
\begin{align}
	& \Bigl| \sum_{k=0}^{\infty} \Cov\bigl(\pi_q f(X_0) , \pi_q g(X_k)\bigr) - \Cov\bigr(f(X_0) , g(X_k)\bigr) \Bigr| \notag\\
	\shoveright&\leq 
	\sum_{k=0}^{k(q)} \bigl| \Cov\bigl(\pi_q f(X_0) -f(X_0) , \pi_q g(X_k)\bigr) \bigr|
	+ \sum_{k=0}^{k(q)} \bigl| \Cov\bigl( f(X_0) , \pi_q g(X_k) - g(X_k)\bigr) \bigr| \label{eq:1805121} \\
	\shoveright&\hspace{3ex}+ \sum_{k=k(q)+1}^{\infty} \bigl| \Cov\bigl(\pi_q f(X_0) -f(X_0) , \pi_q g(X_k)\bigr) \bigr| \label{eq:1805122}\\
	\shoveright&\hspace{3ex}+ \sum_{k=k(q)+1}^{\infty} \bigl| \Cov\bigl( f(X_0) , \pi_q g(X_k) - g(X_k)\bigr) \bigr|. \label{eq:1805123}
\end{align}
Let us treat the terms separately. 
Recall that both $\F$ and $\G$ are uniformly bounded in $\|\cdot\|_\infty$-norm. 
For the term in line \eqref{eq:1805121}, we know by H\"older's inequality, \eqref{eq:pms}, and the fact that $b>1$ that
\begin{align*}
	&\sum_{k=0}^{k(q)} \bigl| \Cov\bigl(\pi_q f(X_0) -f(X_0) , \pi_q g(X_k)\bigr) \bigr|
	+ \sum_{k=0}^{k(q)} \bigl| \Cov\bigl( f(X_0) , \pi_q g(X_k) - g(X_k)\bigr) \bigr| \\
	\shoveright&\ll \sum_{k=0}^{k(q)} \bigl( \| \pi_q f -f \|_s + \| \pi_q g -g \|_s \bigr) \\
	\shoveright&\ll  k(q) 2^{-q} 
	= 2^{-(1-\frac{1}{b})q} \lra 0\quad \text{as}\ q\to\infty,
\end{align*}
where again, we write $x\ll y$ if there is a constant $C\in(0,\infty)$ depending only on global parameters such that $x\leq Cy$. 
For the term in line \eqref{eq:1805122}, by \eqref{eq:cov_phi-f}, \eqref{eq:pms}, and \eqref{eq:pmB} we obtain
\begin{align*}
	&\sum_{k=k(q)+1}^{\infty} \bigl| \Cov\bigl(\pi_q f(X_0) -f(X_0) , \pi_q g(X_k)\bigr) \bigr| \\
	\shoveright&	\le D \| \pi_q f - f\|_\infty  \sum_{k=k(q)+1}^{\infty} \| \pi_q g\|_\C\, \theta^k \\
	\shoveright&	\ll \sum_{k=k(q)+1}^{\infty}\Psi(2^q) \theta^k \lra 0\quad \text{as}\ q\to\infty,
\end{align*}
where we used that $\Psi$ is increasing and condition \eqref{eq:sum_psi_theta} in the last step.
It only remains to show, that the term in line \eqref{eq:1805123} goes to zero as $q\to\infty$. We have
\begin{align}
	&\sum_{k=k(q)+1}^{\infty} \bigl| \Cov\bigl( f(X_0) , \pi_q g(X_k) - g(X_k)\bigr) \bigr|
	\notag\\
	\shoveright&	\leq \sum_{k=k(q)+1}^{\infty} \bigl| \Cov\bigl( f(X_0) , \pi_q g(X_k) - g_k(X_k)\bigr) \bigr| \label{eq:1805124}\\
	\shoveright&	\hspace{3ex} + \sum_{k=k(q)+1}^{\infty} \bigl| \Cov\bigl( f(X_0) , g_k(X_k) - g(X_k)\bigr) \bigr|. \label{eq:2205121}
\end{align}
First, consider the term in line \eqref{eq:1805124}. By \eqref{eq:cov_phi-f}, \eqref{eq:pmB}, and \eqref{eq:pkB} 
\begin{align*}
	&\sum_{k=k(q)+1}^{\infty} \bigl| \Cov\bigl( f(X_0) , \pi_q g(X_k) - g_k(X_k)\bigr) \bigr| 
	\notag\\
	\shoveright&	\ll \sum_{k=k(q)+1}^{\infty} \|f \|_\infty \, \|\pi_q g-g_k\|_\C \, \theta^k
	\notag\\
	\shoveright&	\ll  \, \biggl(  \sum_{k=k(q)+1}^{\infty} \| \pi_q g\|_\C\, \theta^k\biggr)
	\ +\ \biggl( \sum_{k=k(q)+1}^{\infty}  \| g_k \|_\C \, \theta^k \biggr) 
	\notag\\
	\shoveright&	\ll  \biggl( \sum_{k=k(q)+1}^{\infty} \Psi(2^q)\theta^k \biggr)
	\ +\ \biggl( \sum_{k=k(q)+1}^{\infty}  \Psi(k^b) \theta^k \Biggr)\lra 0\quad \text{as}\ q\to\infty,
\end{align*}
where we used that $\Psi$ is increasing and applied condition \eqref{eq:sum_psi_theta} in the last line.
To treat the term in line \eqref{eq:2205121}, we use H\"older's inequality and \eqref{eq:pks}. We obtain
\begin{align*}
	\sum_{k=k(q)+1}^{\infty} \bigl| \Cov\bigl( f(X_0) , g_k(X_k) - g(X_k)\bigr) \bigr| 
	& \ll \sum_{k=k(q)+1}^{\infty} \| g_k-g\|_s \\
	&\ll \sum_{k=k(q)+1}^{\infty} k^{-b}\lra 0 \quad \text{as}\ q\to\infty,
\end{align*}
since $b>1$ and thus $\sum_{k=1}^{\infty} k^{-b}<\infty$, which completes the proof.
\qed


\providecommand{\bysame}{\leavevmode\hbox to3em{\hrulefill}\thinspace}
\providecommand{\MR}{\relax\ifhmode\unskip\space\fi MR }
\providecommand{\MRhref}[2]{%
  \href{http://www.ams.org/mathscinet-getitem?mr=#1}{#2}
}
\providecommand{\href}[2]{#2}


\paragraph{Acknowledgement.}\
 The authors thank the referee for her/his very careful reading of an earlier version of this paper, and for many  thoughtful comments. The suggestions  made by the referee helped to improve the presentation of the paper. 


\end{document}